\documentclass[reqno, 11pt, english]{smfart}

\title{Dynamical behavior of a nondiffusive scheme for the advection equation}

\author{Nina Aguillon}
\email{nina.aguillon@sorbonne-universite.fr}
\address{Sorbonne-Université, CNRS, Universit{\'e} de Paris, INRIA, Laboratoire Jacques-Louis Lions (LJLL), {\'e}quipe ANGE, F-75005 Paris, France}

\author{Pierre-Antoine Guih\'eneuf}
\email{ pierre-antoine.guiheneuf@imj-prg.fr}
\address[Pierre-Antoine Guih{\'e}neuf]{Sorbonne Universit{\'e}, Universit{\'e} Paris Diderot, CNRS, Institut de Math{\'e}matiques de Jussieu-Paris Rive Gauche, IMJ-PRG, F-75005, Paris, France}

\usepackage[english]{babel}
\usepackage{amsfonts}
\usepackage{amsmath}
\usepackage{amssymb}
\usepackage{stmaryrd}
\usepackage{amsthm}
\usepackage{enumerate}
\usepackage{pgf,tikz}
\usetikzlibrary{decorations.pathreplacing, shapes.multipart, arrows, matrix}

\usepackage[colorlinks=true]{hyperref}
\usepackage{psfrag}

\newtheorem{lemme}{Lemma}

\newtheorem{theoreme}[lemme]{Theorem}
\newtheorem{prop}[lemme]{Proposition}

\newtheorem{conj}[lemme]{Conjecture}

\newtheorem{theo}{Theorem}

\theoremstyle{definition}
\newtheorem{definition}[lemme]{Definition}

\theoremstyle{remark}
\newtheorem{rem}[lemme]{Remark}

\newcommand{\eps}{\varepsilon}

\newcommand{\N}{\mathbf{N}}

\newcommand{\R}{\mathbf{R}}

\newcommand{\Z}{\mathbf{Z}}

\newcommand{\varep}{\varepsilon}

\renewcommand{\mod}{\, \mathop{mod} \,}

\definecolor{mygreen}{HTML}{009000}
\definecolor{mypurple}{HTML}{900090}

\begin{document}

\maketitle

\begin{abstract}
We study the long time behaviour of a dynamical system strongly linked to the anti-diffusive scheme of Despr\'es and Lagoutiere for the $1$-dimensional transport equation. This scheme is nondiffusive in the sens that discontinuities are not smoothened out through time. Numerical simulations indicates that the scheme error's is uniformly bounded with time. We prove that this scheme is overcompressive when the Courant--Friedrichs--Levy number is $1/2$: when the initial data is nondecreasing, the approximate solution becomes a Heaviside function. In a special case, we also understand how plateaus are formed in the solution and their stability, a distinctive feature of the Despr\'es and Lagoutiere scheme.
\end{abstract}

\setcounter{tocdepth}{1}
\tableofcontents

\section*{Introduction}

The numerical approximation of the solution to the $1$-dimensionnal transport equation with a constant velocity $V>0$ has received a lot of attention for a long time, and still continues to do. One of the reason is that this equation, namely
\begin{equation} \label{eq:transport}
  \begin{cases}
   \partial_t u (t,x) + V \partial_x u(t,x)= 0 & \forall t>0, \ \forall x \in \R \\
   u(0,x)=u^0(x) & \forall x \in \R
  \end{cases}
   \end{equation} 
is very simple and well understood, and is at the same time a fundamental example in the much larger class of conservation laws. In the multidimensionnal setting with space and time dependent velocity fields,~\eqref{eq:transport} is of crucial importance for practical applications, as it represents the passive advection of the quantity $u$.
For this class of equation it is of crucial importance to have reliable and accurate numerical schemes, able to capture the exact solution $u(t,x)=u^0(x-Vt)$ of~\eqref{eq:transport}, even when $u^0$ is discontinuous.

One of the simplest schemes to approximate~\eqref{eq:transport} is the so-called upwind scheme. Fixed time step $\Delta t>0$ and space step $\Delta x>0$ are given, and the real line $\R$ is separated in intervals of size $\Delta x$, with midpoints $x_i=i \Delta x$, and left extremity $x_{i-1/2}=x_i - \frac{\Delta x}{2}$. The scheme is initialized with $u_j^0=u^0(x_j)$ if $u^0$ is $\mathcal{C}^1$-regular or with
\begin{equation} \label{eq:ini}
 u_j^0= \frac{1}{\Delta x} \int_{x_{i-1/2}}^{x_{i+1/2}} u^0(x) dx
\end{equation}
when $u^0$ only  has bounded variations.
The approximate solution $(u_j^{n+1})_{j \in \Z}$ at time $(n+1)\Delta t$ is obtained from $(u_j^n)_{j \in \Z}$, the approximate solution at time $n \Delta t$, by
\begin{equation} \label{upwind}
 \frac{u_j^{n+1}-u_j^n}{\Delta t} + V \frac{u_j^n-u_{j-1}^n}{\Delta x} = 0, \qquad \forall j \in \Z, \ \forall n \in \N.
\end{equation}
An interpretation is the following. At time $n \Delta t$, define a piecewise constant function by
\[u_{\Delta x}^n(x) = u_j^n \ \text{ if } x \in \left[x_{j-1/2},  x_{j+1/2}\right). \]
Translate it to the right of a distance $V \Delta t$, i.e. consider $v$ the exact solution of~\eqref{eq:transport} with initial data $u_{\Delta x}^n$ at time $\Delta t$. 
If the solution does not cross more than a cell, namely if $V \Delta t < \Delta x$, then
\[ \begin{aligned}
    \int_{x_{i-1/2}}^{x_{i+1/2}} v(x) dx &= \int_{x_{i-1/2}}^{x_{i+1/2}} u_{\Delta x}^n( x- V \Delta t) dx \\
    &=  V \Delta t u_{j-1}^n + (\Delta x - V \Delta t) u_j^n = \Delta x u_j^{n+1}. 
   \end{aligned}
\]
It is possible to prove that the resulting scheme converges towards the exact solution.
\begin{theo} \label{thmconv}
Suppose that the Courant--Friedrichs--Levy number $\frac{V \Delta t}{\Delta x}$ is fixed in the interval $(0,1)$.
\begin{itemize}
 \item If $u^0$ if $\mathcal{C}^2$-regular and with the initialization $u_j^0=u^0(x_j)$, there exists a constant $C$ such that
  \[ \max_{j \in \Z} \left| u_j^n -   u(n \Delta t, x_j)  \right| \leq C n \Delta t \Delta x. \]
  \item If $u^0$ has total variation and with the initialization~\eqref{eq:ini}, there exists a constant $C$ such that
  \[\forall n \in \N, \qquad  \Delta x \sum_{j \in \Z} \left| u_j^n -  \frac{1}{\Delta x}  \int_{x_i-\Delta x /2}^{x_i+\Delta x /2} u(n \Delta t, x) dx  \right| \leq C  \sqrt{n \Delta t \Delta x}. \]
\end{itemize}
\end{theo}

This theorem means that \emph{if the final time $T= n \Delta t$ is fixed} and if $\Delta t$ and $\Delta x$ both tend to zero by keeping the ratio $\frac{V \Delta t}{ \Delta x}$ fixed and smaller than $1$, then the approximate solution converges towards the exact solution at rate $1$ or~$1/2$ (depending on the regularity of $u^0$), for the $L^\infty$ and $L^1$ norm respectively. The error in time grows as $T$ or $\sqrt{T}$.

A considerable effort has been made over the last decades to improve the rate of convergence.  For linear schemes, estimates like
\[ ||u_{\Delta x} - u ||_{L^\infty} \leq C(T) \Delta x^p \quad \text{or } \quad ||u_{\Delta x} - u ||_{L^1}  \leq C(T) \Delta x^{\frac{p}{p+1}} \]
have been proven in~\cite{D1} and~\cite{D2}, for regular and BV initial data respectively; Theorem~\ref{thmconv} is a special case of this result. Nonlinear schemes for~\eqref{eq:transport} are widely used, because it is the only way to obtain methods that are of order larger than $2$ and that verify a discrete maximum principle. For a description of the most popular methods for the linear advection equation, see~\cite{Leveque}.

\subsection*{The antidiffusive scheme of Despr\'es and Lagouti\`ere} \label{SDL}

Among all the schemes available for~\eqref{eq:transport}, the scheme introduced by Despr\'es and Lagouti\`ere in~\cite{DL02}  has the property of having an error that does not grow with time. This property has been verified numerically but is still a conjecture, and this paper is a step toward its proof.

Their main idea is to reverse the average step of the upwind scheme~\eqref{upwind}, by considering that each value $u_j^n$ comes from an average of a discontinuity joining $u_{j-1}^n$ to $u_{j+1}^n$ located somewhere inside the cell (recall that at time $n \Delta t$, the approximate solution is constant equal to $u_j^n$ on the interval $[x_{i-1/2},x_{i+1/2})$). This scheme can be decomposed in three steps:
\begin{enumerate}
 \item  In $[x_{i-1/2},x_{i+1/2})$, replace $u_j^n$  by a piecewise constant map of the form
\[ (u_{rec}^n)_{|[x_{i-1/2},x_{i+1/2})} : x\longmapsto \left\{\begin{array}{l}
u_{j-1}^n \quad \text{if}\quad  x_{i-1/2} \le x < x_{i-1/2} + d_j^n\\
u_{j+1}^n \quad \text{if}\quad x_{i-1/2} + d_j^n \le x < x_{i+1/2}
\end{array}\right.\]
The discontinuity is placed at a distance $d_j^n \in [0, \Delta x]$ of the left extremity of the cell, in such a way that the total mass inside the cell is preserved, i.e. 
 \[ \Delta x\, u_j^n= d_j^n u_{j-1}^n + (\Delta x-d_j^n) u_{j+1}^n. \]
  If this is not possible, do nothing, i.e. $(u_{rec}^n)_{|[x_{i-1/2},x_{i+1/2})} = u_j^n$.
 \item Compute the exact solution of~\eqref{eq:transport} with initial data $u^n_{rec}$ at time $\Delta t$, which is nothing but $x \mapsto u^n_{rec}(x-V \Delta t)$.
 \item Define $u_j^{n+1}$ as the average of this exact solution on $[x_{j-1/2}, x_{j+1/2}]$:
 \[ u_j^{n+1}= \dfrac{1}{\Delta x} \int_{x_{j-1/2}}^{x_{j+1/2}} u_{rec}^n(x- V \Delta t) dx. \]
\end{enumerate}
This process is illustrated on Figure~\ref{F:RecScheme}, where regions of same colors are of equal areas.
This interpretation in terms of discontinuous reconstruction is equivalent to the original presentation of~\cite{DL02} and is presented in~\cite{BCLL08}. The scheme is initialized with~\eqref{eq:ini}.

\begin{figure}[h!tp]
\begin{psfrags}
\psfrag{uj}{$\textcolor{mygreen}{u_j^n}$}
\psfrag{uj-1}{$\textcolor{mygreen}{u_{j-1}^n}$}
\psfrag{uj+1}{$\textcolor{mygreen}{u_{j+1}^n}$}
\psfrag{ujl}{}
\psfrag{ujr}{}
\psfrag{unrec}{$\textcolor{mypurple}{u_{rec}^n}$}
\psfrag{ujb}{$\textcolor{mygreen}{u_j^n}$}
\psfrag{uj-1b}{$\textcolor{mygreen}{u_{j-1}^{n+1}}$}
\psfrag{uj+1b}{$\textcolor{mygreen}{u_{j+1}^{n+1}}$}
\psfrag{ujb}{$\textcolor{mygreen}{u_{j}^{n+1}}$}
\psfrag{unrecb}{$\textcolor{mypurple}{u_{rec}^n}$}
\psfrag{d}{$d_j^n$}
\psfrag{CFL}{$V \Delta t$}
\psfrag{xj-1/2}{$x_{j-1/2}$}
\psfrag{xj+1/2}{$x_{j+1/2}$}
 \includegraphics[width=\linewidth]{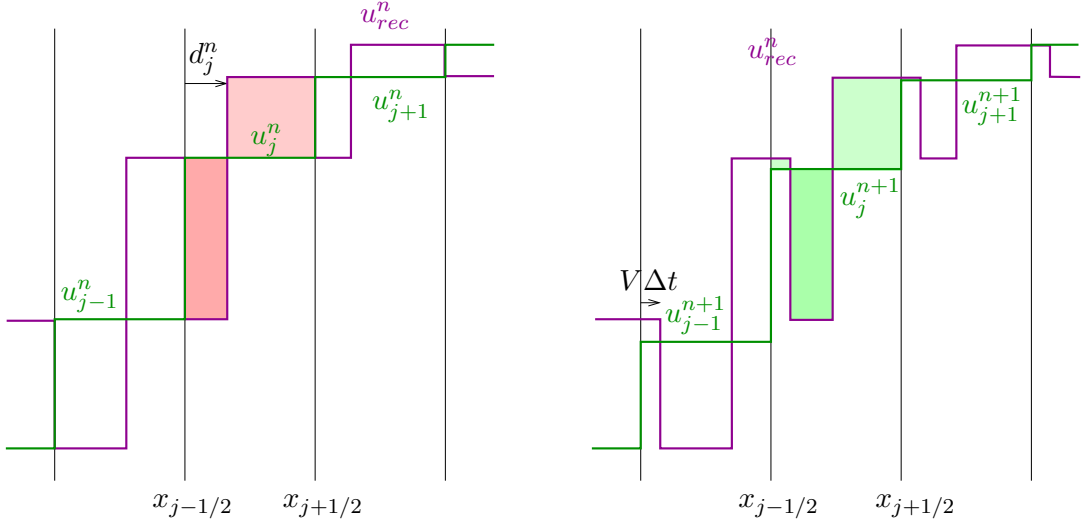} 
 \caption{Left: initial data $(u_j^{n})_j$ (green) and reconstruction $u^n_{rec}$ (purple). Right: the solution at time $n+1$ (green) is the $L^2$-projection of the reconstruction translated of $V \Delta t$ (purple).}  \label{F:RecScheme}
\end{psfrags}
\end{figure}

The property that makes this scheme unique is that it is exact for a large class of initial data (the vast majority of schemes are exact only for constant initial data).

\begin{prop}[Despr\'es, Lagouti\`ere, Theorem 3 of~\cite{DL02}] \label{attractor}
 Suppose that $u^0$ is piecewise constant, with plateaus of width larger than $3 \Delta x$. Then
 \[ \forall n \in \N, \ \forall j \in \Z, \qquad u_{j}^{n}= \int_{x_{j-1/2}}^{x_{j+1/2}} u^0(x-V n \Delta t) dx .\]
\end{prop}

\begin{proof}
The proof can be found in~\cite{DL02}. It boils down to proving that $u^n_{rec}$ is the $L^2$-projection on the grid of the exact solution at time $n \Delta t$, which is $x \mapsto u^0(x-V n \Delta t)$. The computation is a bit tedious but the idea is simple. Let us focus on the initialization. As plateaus are wider than $3 \Delta x$, each of them contains at least two consecutive cells which have the same value after the initialization~\eqref{eq:ini}. It ensures that the first step of the reconstruction scheme is successful only on cells containing a discontinuity of the initial data $u^0$, and thus that $u^0=u^0_{rec}$.
\end{proof}

Numerically it is observed that this class of initial data behaves as an attractor. Plateaus are created in the first time steps and are then advected exactly. We have the following conjecture, that looks very much like Theorem~\ref{thmconv} but with a time independent constant.

\begin{conj}\label{conjconv}
Let $u^0$ be a function with total bounded variations to which we associate the initialization $u_j^0= \frac{1}{\Delta x} \int_{x_{i-1/2}}^{x_{i+1/2}} u^0(x) dx$. Suppose that the ratio $\lambda=\frac{\Delta t}{V \Delta x}$ is kept fixed and belongs to $(0,1] \setminus \{1/2\}$. Then there exists a constant $C>0$, depending only on $u^0$ and $\lambda$ such that 
\[ \forall n \in \N, \qquad \Delta x \sum_{j \in \Z} \left| u_j^n -  \frac{1}{\Delta x}  \int_{x_i-\Delta x /2}^{x_i+\Delta x /2} u(n \Delta t, x) dx  \right| \leq C \sqrt{\Delta x}. \]
\end{conj}

This conjecture states that there is a global attractor $A$ for bounded increasing configurations, made of solutions whose reconstructions have plateaus of width bigger than 3 (see Figure \ref{FigAttrac}). In other words, we expect that for any bounded increasing initial data $(u_j^0)$ and for any $\eps>0$, there is a solution $(\tilde u_j^0)\in A$ and an integer $N\in\N$ such that for any $n\ge N$, one has $\|(u_j^n)_j - (\tilde u_j^n)_j\|_\infty \le \varep$. Note that the elements of $A$ are almost periodic. For more details about the concept of attractor, see \cite{MR818833, MR790735}, \cite{Milnor:2006}.

\begin{figure}
\begin{center}
\begin{tikzpicture}[scale=.9]

\draw[color=blue!80!black, thick] (0,0) to[bend left] (3,2)node[right]{$A$};
\draw[->, color=gray!50!black] (-1.2,1) -- (-.1,.3);
\draw[->, color=gray!50!black] (-.2,2) -- (.6,1.2);
\draw[->, color=gray!50!black] (1.2,3) -- (1.7,2.1);
\draw[->, color=gray!50!black] (3,.3) -- (2.5,1.3);
\draw[->, color=gray!50!black] (2,0) -- (1.3,.8);
\draw[->, color=gray!50!black] (1.4,-.3) -- (.6,.2);

\end{tikzpicture}
\end{center}
\caption{\label{FigAttrac}Picture of an attractor $A$: any solution eventually approaches a solution inside the set $A$ of solutions whose reconstructions have plateaus of width bigger than 3}
\end{figure}
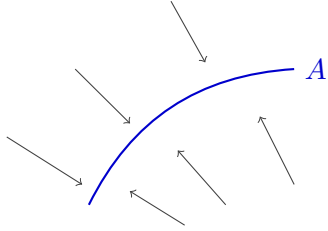

\bigskip

The main result of our work is to prove that this conjecture does not hold if $\lambda=1/2$. We roughly speaking prove the following result (see Theorem~\ref{thmover} for a precise statement).

\begin{theo} \label{conv1/2}
Suppose that the initial sequence $(u_j^0)_{j \in \Z}$ strictly increases from a constant value $\alpha$ to another constant value $\beta$. Then for all $n\in\N$ large enough, there exists an integer $j_\infty^n\in\Z$ such that $(u_j^n)_{j \in \Z}$ is a discrete Heaviside function, that is,
  $$ u_j^n=\begin{cases}
            \alpha & \text{ if }j < j_\infty^n \\ 
            \beta & \text{ if } j > j_\infty^n
           \end{cases} $$
 \end{theo}

We also get a similar statement for initial conditions that are ``half infinite staircases'' in Proposition \ref{PropEscalier}.

In a last part of the article, we prove a special case of Conjecture \ref{conjconv} in the case where the grid is alternatively shifted to the left and to the right of a parameter $\lambda\neq 1/2$ (see Proposition \ref{PropConvExp} for a precise statement).

\begin{theo}
For any $\lambda\in(0,1)\setminus \{1/2\}$, there exists a nonempty open set of initial data for which $(u_i^n)_i$ converges uniformly exponentially fast towards a limit configuration.
\end{theo}

In Section~\ref{notation}, we simplify the problem and present important lemmas. Some numerical illustrations of Conjecture~\ref{conjconv} and Proposition~\ref{attractor} are given in Section~\ref{simu}. Theorem~\ref{conv1/2} is proven in Section~\ref{lambda=1/2}. Eventually in Section~\ref{5conf}, a particular case illustrating the exponential convergence toward solutions with plateaus is studied. This result is far from being a complete proof of Conjecture~\ref{conjconv} which remains open at the moment.

\section{A related shifted grids dynamical process} \label{notation}

In order to simplify the analysis while retaining the most important aspects, we do the following modifications. First, we set $V=1$, $\Delta x=1$ and $x_j=j$. More importantly, instead of moving the reconstructed solution of $\lambda$ to the right, we shift the grid of $\lambda$ to the left for odd iteration in time and of $\lambda$ to the right for even iteration in time. The advantage of shifting the grid alternatively to the left and to the right is obviously that we end up with the same grid after two iterations. If $\lambda=1/2$, there is no difference with the case where the grid is always shifted to the left. For any real number $a$ we denote by $\mathcal{C}_a$ the interval centered around $a$ of size $1$: $\mathcal{C}_a=(a-1/2, a+1/2)$.

The structure of the scheme follows the same guidelines than the Despr\'es and Lagouti\`ere scheme presented in the previous section. The process is initialized with the sequence $(u_j^0)_{j \in \Z}$ given by~\eqref{eq:ini}.
For odd iteration in time the process is centered on integer points $j$ at the beginning on the time step and the grid is shifted of $\lambda$ to the left. 
The three steps are the following
\begin{enumerate}
 \item \emph{Reconstruction step.} Compute the distance $d_j^{2n}$ from the \emph{right} interface $j+1/2$ such that 
 \[(1-d_j^{2n}) u_{j-1}^{2n} + d_j^{2n} u_{j+1}^{2n} = u_j^{2n}.\]
One gets
 \begin{equation}\label{EqAlpha}
d_j^{2n} = \frac{u_{j}^{2n} - u_{j-1}^{2n}}{u_{j+1}^{2n} - u_{j-1}^{2n}},
\end{equation}
and we set arbitrarily $d_j^{2n}=-1$ if it is not defined.
Then, define
 \[ u_{j,L}^{2n}= \begin{cases}
                u_{j-1}^{2n} & \text{ if } 0<d_j^{2n}<1, \\
                u_{j}^{2n} & \text{ otherwise, }  \\
               \end{cases}
\qquad
     u_{j,R}^{2n}= \begin{cases}
                u_{j+1}^{2n} & \text{ if } 0<d_j^{2n}<1, \\
                u_{j}^{2n} & \text{ otherwise. }  \\
               \end{cases}
\]
The reconstructed solution at iteration $2n$ is obtained as
\[ u^{2n}_{rec}(x)= \sum_{j \in \Z} \left( u_{j,L}^{2n} \mathbf{1}_{d_j^{2n}<(j+1/2)-x<1} + u_{j,R}^{2n} \mathbf{1}_{0<(j+1/2)-x<d_j^{2n}} \right) \mathbf{1}_{x \in \mathcal{C}_{j}}.\]
\item \emph{Shifting.} Shift the grid of $\lambda$ to the left and define
\[ u_{j-\lambda}^{2n+1}= \int_{\mathcal{C}_{j-\lambda}} u^{2n}_{rec}(x) \, dx \]
and $\displaystyle u^{2n+1}(x)=\sum_{j \in \Z} u_{j-\lambda}^{2n} \mathbf{1}_{x \in \mathcal{C}_{j-\lambda}}. $
\end{enumerate}
At the beginning of an even iteration in time, the cells are centered around the points $(j-\lambda)_{j \in \Z}$, and we follow the same process but move the grid to the right:
\begin{enumerate}
 \item \emph{Reconstruction step.} Compute the distance $d_{j-\lambda}^{2n+1}$ from the \emph{left} interface such that 
 \[d_{j-\lambda}^{2n+1} u_{j-1-\lambda}^{2n+1} + (1-d_{j-\lambda}^{2n+1}) u_{j+1-\lambda}^{2n+1} = u_{j-\lambda}^{2n+1}.\]
 If does not exists, set $d_{j-\lambda}^{2n+1}=-1$. Then define
 \[ u_{j-\lambda,L}^{2n}= \begin{cases}
                u_{j-1-\lambda}^{2n+1} & \text{ if } 0<d_{j-\lambda}^{2n+1}<1, \\
                u_{j-\lambda}^{2n+1} & \text{ otherwise, }  \\
               \end{cases}
\
     u_{j-\lambda,R}^{2n}= \begin{cases}
                u_{j+1-\lambda}^{2n+1} & \text{ if } 0<d_{j-\lambda}^{2n+1}<1, \\
                u_{j-\lambda}^{2n+1} & \text{ otherwise. }  \\
               \end{cases}
\]
The reconstructed solution at iteration $2n+1$ is 
\[ u^{2n+1}_{rec}(x)= \sum_{j \in \Z} \left( u_{j,L}^{2n} \mathbf{1}_{0<x-(j-\lambda-1/2)<d_{j-\lambda}^{2n+1}} + u_{j,R}^{2n} \mathbf{1}_{d_{j-\lambda}^{2n+1}<x-(j-\lambda-1/2)<1} \right)\mathbf{1}_{x \in \mathcal{C}_{j-\lambda}}.\]
\item \emph{Shifting.} Shift the grid of $\lambda$ to the right and define
\[ u_{j}^{2n+2}= \int_{\mathcal{C}_{j}} u^{2n+1}_{rec}(x) \, dx \]
and $u^{2n+2}(x)=\sum_{j \in \Z} u_{j}^{2n+2} \mathbf{1}_{x \in  \mathcal{C}_{j}}. $
\end{enumerate}

The notations are gathered on Figure~\ref{F:RecScheme2}.
In our analysis we restrict our attention to the case of nondecreasing initial data. 
If the initial data is increasing, this property is inherited at each time step.

\begin{figure}[h!tp]
\begin{psfrags}
\psfrag{uj}{$\textcolor{mygreen}{u_j^{2n}}$}
\psfrag{uj-1}{$\textcolor{mygreen}{u_{j-1}^{2n}}$}
\psfrag{uj+1}{$\textcolor{mygreen}{u_{j+1}^{2n}}$}
\psfrag{ujl}{$\textcolor{mypurple}{u_{j,L}^{2n}}$}
\psfrag{ujr}{$\textcolor{mypurple}{u_{j,R}^{2n}}$}
\psfrag{unrec}{$\textcolor{mypurple}{u_{rec}^{2n}}$}
\psfrag{ujb}{$\textcolor{mygreen}{u_j^{2n}}$}
\psfrag{uj-1b}{$\textcolor{mygreen}{u_{j-1}^{2n+1}}$}
\psfrag{uj+1b}{$\textcolor{mygreen}{u_{j+1}^{2n+1}}$}
\psfrag{ujb}{$\textcolor{mygreen}{u_{j}^{2n+1}}$}
\psfrag{unrecb}{$\textcolor{mypurple}{u_{rec}^{2n}}$}
\psfrag{d}{$d_j^{2n}$}
\psfrag{l}{$\lambda$}
\psfrag{j}{$j-1/2$}
\psfrag{j+1}{$j+1/2$}
\psfrag{j-l}{$j-1/2-\lambda$}
\psfrag{j+1-l}{$j+1/2-\lambda$}
\psfrag{Cj}{$\mathcal{C}_j$}
\psfrag{Cj-l}{$\mathcal{C}_{j-\lambda}$}
 \includegraphics[width=\linewidth]{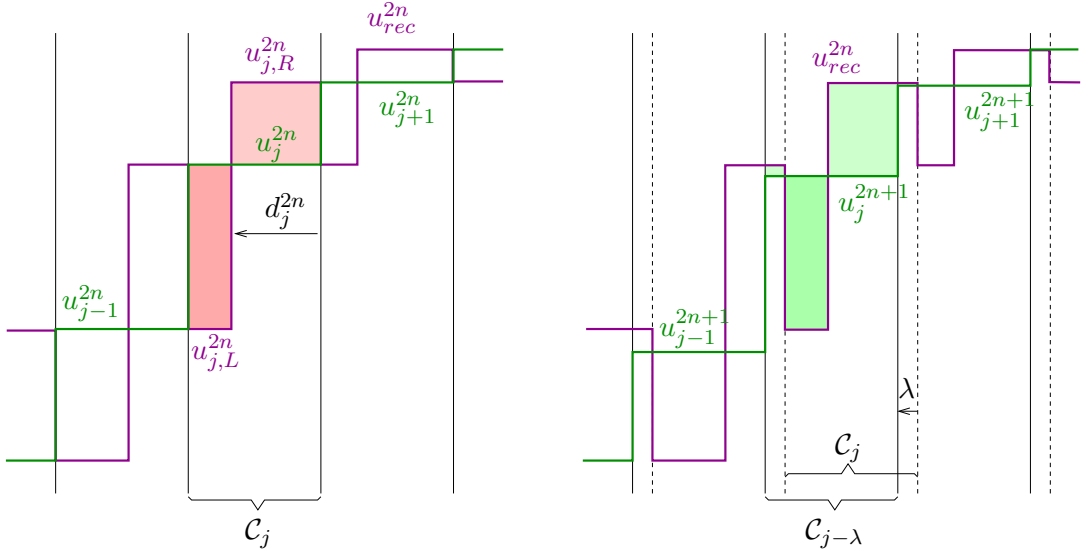} 
 \caption{The process where the grid is shifted to the left (odd iterations in time).}\label{F:RecScheme2}
\end{psfrags}
\end{figure}

\begin{lemme} Suppose that $x \mapsto u^n(x)$ is non-decreasing. Then $x \mapsto u^{n+1}(x)$ is also non-decreasing.
\end{lemme}

\begin{proof} \label{monotonicity}
Suppose that $n$ is even. Let us denote by $u_{rec}^n$ the reconstruction map, which is increasing. Then the mean value function
\[ r \mapsto \dfrac{1}{r} \int_{j-1/2}^{j-1/2+r} u_{rec}^n(x) \, dx \]
increases on $(0,1)$ from $u_{j-1}^n$ for $r=0$ to $u_{j}^n$ for $r=1$. One equally proves that the mean of $u_{rec}^n$ on $[j-1/2-r,j-1/2]$, $r \in (0,1)$ belongs to $[u_{j-1}^n,u_{j}^n]$. Thus, the mean of $u_{rec}^n$ on $\mathcal{C}_{j-\lambda}$ belongs to $[u_{j-1}^n,u_{j}^n]$, in other words $u_{j-\lambda}^{n+1} \in [u_{j-1}^n,u_{j}^n]$ which is smaller than $u_{j+1-\lambda}^{n+1} \in [u_{j}^n,u_{j+1}^n]$.
When $n$ is odd, we similarly prove that $u_j^{n+1}$ belongs to $[u_{j-\lambda}^n, u_{j+1-\lambda}^n]$.
\end{proof}
The general case follows from the study of the non decreasing one, as explained by the following lemma.
\begin{lemme}
Suppose that $(u_j^0)_{j\in \Z}$ is given. Without loss of generality, we suppose that $u_0^0=0$. Let us define $(v_j^0)_{j\in \Z}$, $(w_j^0)_{j\in \Z}$ by $v_0^0=w_0^0=0$ and
 \[ \begin{matrix}
 v_{j+1}^0= \begin{cases}
                v_j^0+ (u_{j+1}^0-u_j^0) & \text{ if } u_{j+1}^0>u_j^0 \\
                v_j^0 & \text{ if } u_{j+1}^n \leq u_j^n
               \end{cases} \\
w_{j+1}^0= \begin{cases}
                w_j^0 & \text{ if } u_{j+1}^0>u_j^0 \\
                w_j^0 + (u_{j+1}^0-u_j^0) & \text{ if } u_{j+1}^0 \leq u_j^0
               \end{cases}     
    \end{matrix}
\]
Then for all $n \geq 0$, 
$$(u_{j- \lambda (n \mod2)}^{n})_{j \in  \Z}=(v_{j- \lambda (n \mod 2)}^{n})_{j \in  \Z}+(w_{j- \lambda (n \mod2)}^{n})_{j \in  \Z}.$$
\end{lemme}

\begin{proof}
The proof boils down to show that for all $n$, $u_{rec}^n=v_{rec}^n+w_{rec}^n$. Let us prove it for $n=0$.
We distinguish cases depending on the relative positions of $u_{j-1}^0$, $u_j^0$ and $u_{j+1}^0$.
 \begin{itemize}
  \item If $u_{j-1}^0 = u_{j}^0  = u_{j+1}^0$, the reconstruction on cell $\mathcal{C}_j$ is constant equal to~$u_j^0$. It is clear that $v_{j-1}^0 = v_{j}^0  = v_{j+1}^0$ and $w_{j-1}^0 = w_{j}^0  = w_{j+1}^0$, and thus
  \[ (v_{rec}^0)_{|\mathcal{C}_j} + (w_{rec}^0)_{|\mathcal{C}_j} = v_j^0 + w_j^0 = u_j^0 = (u_{rec}^0)_{|\mathcal{C}_j}. \]
  \item If $u_{j-1}^0 \leq u_{j}^0  \leq u_{j+1}^0$ with one strict inequality, a discontinuity is reconstructed in cell $\mathcal{C}_j$ at a distance $j_j^0=\frac{u_{j}^0-u_{j-1}^0}{u_{j+1}^0-u_{j-1}^0}$ of the right interface and
  \[ (u_{rec}^0)_{|\mathcal{C}_j}(x)= u_{j-1}^0 \mathbf{1}_{d_j^0<j+1/2-x<1} + u_{j+1}^0 \mathbf{1}_{0<j+1/2-x<d_j^0}.\]
  In this case we have
  \[ (v_{j-1}^0,v_j^0, v_{j+1}^0) =  v_j^0-u_j^0 + (u_{j-1}^0,u_j^0, u_{j+1}^0) \]
  and thus $(v_{rec}^0)_{|\mathcal{C}_j}= v_j^0-u_j^0 + (u_{rec}^0)_{|\mathcal{C}_j}$. Moreover $w_{j-1}^0=w_{j}^0=w_{j+1}^0$, thus $(w_{rec}^0)_{|\mathcal{C}_j}= w_j^0$ and the results follows.
  \item  If $u_{j-1}^0 \geq u_{j}^0$ and $u_j^0  \leq u_{j+1}^0$ with one strict inequality, then $d_j^0$ does not belong to $(0,1)$ and thus $(u_{rec}^0)_{|\mathcal{C}_j}=u_j^0$. On the other hand $v_j^0=v_{j+1}^0$ and $w_{j-1}^0=w_j^0$ which yields $(v_{rec}^0)_{|\mathcal{C}_j}=v_j^0$ and $(w_{rec}^0)_{|\mathcal{C}_j}=v_j^0$.
  \item The other cases are treated similarly by exchanging the roles of $v$ and $w$.
 \end{itemize}

\end{proof}

\section{Numerical simulations} \label{simu}
We now give some numerical illustrations of the long time behavior of the scheme and the influence on the parameter $\lambda$.

\subsection{Illustration of Theorem~\ref{thmconv} and Conjecture~\ref{conjconv}}
To begin with, we consider the smooth $1$-periodic initial data defined by
\begin{equation} \label{ID1}
  \forall x \in [0,1], \  u^0(x)= \cos(2 \pi x) \sin(10 \pi x)
\end{equation}
and we compare three classical schemes for the transport equation~\eqref{eq:transport}:
\begin{itemize}
 \item the upwind scheme~\eqref{upwind}, which is linear and first order;
 \item the Lax--Wendroff scheme
 $$u_j^{n+1}=u_j^n-  \dfrac{V\Delta t}{2 \Delta x} (u_{j+1}^n-u_{j-1}^n) +  \dfrac{V^2\Delta t^2}{2 \Delta x^2} (u_{j+1}^n-2 u_j^n+u_{j-1}^n),$$
  which is linear and first order;
 \item the  Despr\'es and Lagouti\`ere scheme explained in the introduction, which is first order and nonlinear;
\end{itemize}
We set $V=1$ and a final time of $T=10$. The space interval $[0,1]$ is discretized with $M$ cells and the time step is related to the space step $\Delta x= 1/M$ by $\Delta t = \frac{0.4 \Delta x}{V}$, which ensures stability and convergence of the schemes (this implies that $\lambda=0.4$). We are interested in the evolution of the $L^\infty$-error
$$ Err(n)= \max_{j\in \{1, \cdots, M\}} \left|u_j^n- u^0(x_j-V n \Delta t)\right|.$$
\begin{figure}[h!] \label{F:castest1}
 \includegraphics[clip=true, trim=3cm 1cm 3cm 1cm, width=\linewidth]{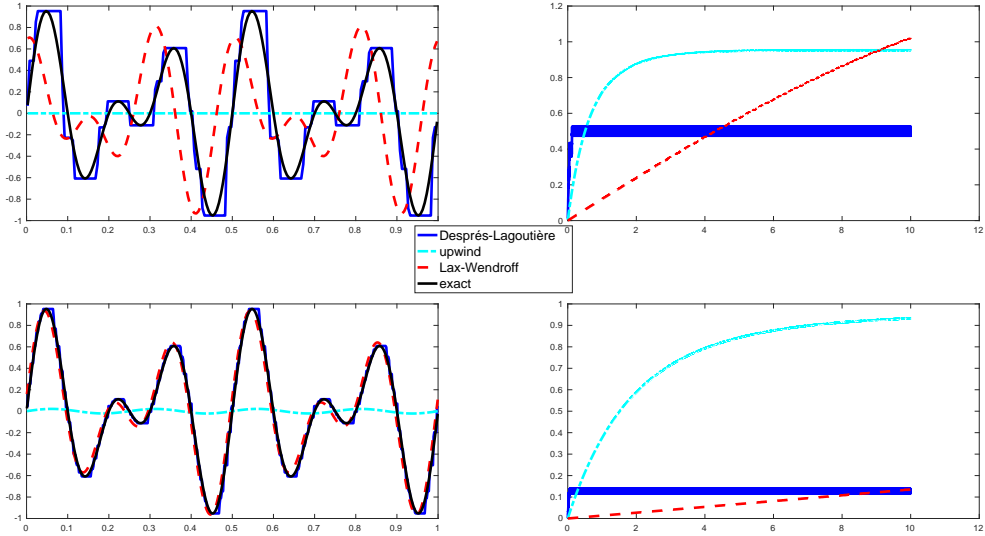}
 \caption{Left: approximate solution after $10$ periods for the initial data~\eqref{ID1} for different schemes with $M=100$ (top, $5 \, 000$ time steps) and $M=600$ (bottom, $15 \, 000$ time steps). Right: evolution of the $L^\infty$-error with time. The blue line looks thicker because it has high frequency oscillations.}
\end{figure}

The results for $M=200$ and $M=600$ are given on Figure~\ref{F:castest1}. We see that the upwind scheme is so diffusive that all oscillations are flatened and the approximate solution is almost constant. The Lax-Wendroff scheme is second order and thus much less diffusive, however the approximate solution is not acceptable for $M=100$. With the reconstruction scheme, stairs appear in the first iterations in time and are then advected exactly. It is the only scheme for which the maximum value does not decrease with time. As expected, the result are better for $M=600$. With finer and finer meshes, we could illustrate the validity of Theorem~\ref{thmconv} on the time interval $[0,15]$. However, whatever the value of $M$ we can reproduce Figure~\ref{F:castest1} by increasing the final time $T$.

\subsection{Influence of $\lambda$}
The initial data is now $1.5$-periodic with
\begin{equation} \label{ID2}
  \forall x \in [-0.3,1.2], \  u^0(x)= \begin{cases}
                                        -1 & \text{ if } -0.3 \leq x \leq 0 \\
                                        \sin(\pi x-\pi/2) & \text{ if } 0 \leq x \leq 1 \\
                                        1 & \text{ if } 1 \leq x \leq 1.2 \\
                                       \end{cases}
\end{equation}
It contains a discontinuity at $x=1.2$ and a smooth part in the interval $[0,1]$.

On Figure~\ref{FSin}, we plot the result at time $22.5$ (which corresponds to $15$ periods) for the original scheme of Despr\'es and Lagouti\`ere described in the introduction and for CFL numbers $\lambda=\frac{V \Delta t}{\Delta x}$  of $0.47$, $0.48$, $0.49$ and $0.5$ the critical value. We took $M=100$ and $V=1$.

\begin{figure}[h!] 
 \includegraphics[clip=true, trim=3cm 0cm 3cm 0cm, width=\linewidth]{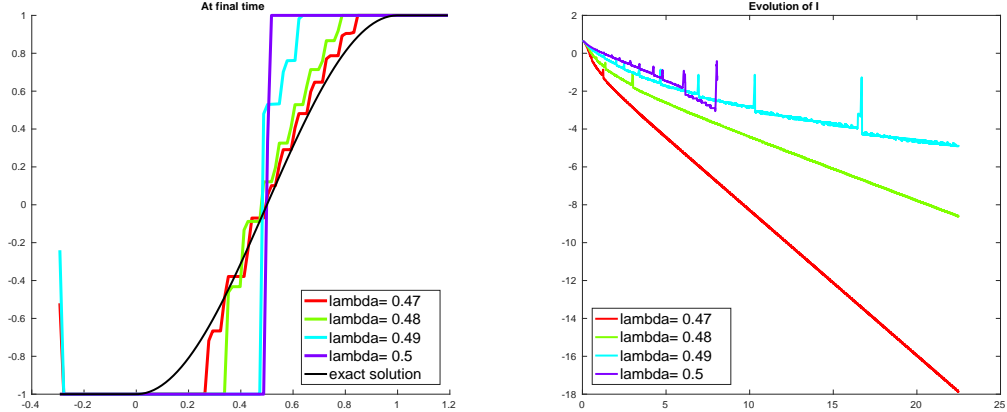}
 \caption{Approximate solution after $15$ periods for the initial data~\eqref{ID2} and for different CFL number, when the grid is fixed. Left: solution at the final time, right: evolution of the quantity $I$ in logarithmic scale.} \label{FSin}
\end{figure}

The final time is large enough to observe the long time behavior of Theorem~\ref{conv1/2} when $\lambda=1/2$: the approximate solution is an Heaviside function. The closest $\lambda$ is close to $1/2$, the fewer steps there is. On the right of this Figure, we plot the quantity
$$ I(n)= \sum_{j=1}^M\min(|u_{j-1}^n-u_j^n|, |u_{j}^n-u_{j+1}^n|,  |u_{j+1}^n-u_{j+2}^n|) $$
(with periodic boundary conditions $u_0^n=u_M^n$, $u_{M+1}^n=u_0^n$ and $u_{M+2}^n=u_1^n$).
This quantity  is null if $(u_j^n)$ is piecewise constant with plateaus of width larger than $3$ cells. Intermediate values between two plateaus are allowed.
It somehow illustrate Proposition~\ref{attractor}.

The results of the same simulation for the related scheme of Section~\ref{notation}, where the grid is shifted alternatively to the left and to the right, are given on Figure~\ref{FSinStag}. We see that the results are more symmetrical and that the convergence is faster.
\begin{figure}[h!] 
 \includegraphics[clip=true, trim=3cm 0cm 3cm 0cm, width=\linewidth]{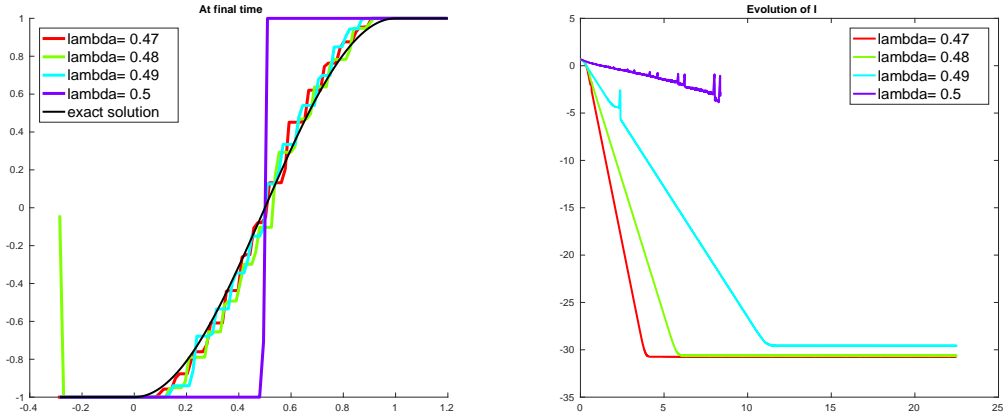}
 \caption{Approximate solution after $15$ periods for the initial data~\eqref{ID2} for different CFL number, when the grid is fixed. Left: solution at the final time, right: evolution of the quantity $I$ in logarithmic scale.} \label{FSinStag}
\end{figure}

\begin{rem}
 The long time behavior observed here for $\lambda=1/2$ is not in contradiction with Theorem~\ref{thmconv}. Indeed for this final time, when the number of cells $M$ is large enough, stairs do not have time to completely merge together and the scheme converges, in the sense that the approximate solution \emph{at time $T$} is closer and closer to the exact one.
\end{rem}

\section{The symmetric case of a half cell shift} \label{lambda=1/2}

In this section we study the long time behavior of the scheme when $\lambda=1/2$ with two types of increasing initial data.

\subsection{Overcompressivity for nonnegative and compactly supported jumps}  \label{overcompressivity}

First we suppose that the initial data consists in a finite succession of strictly positive jumps. We prove that after a finite number of iterations, the numerical solution contains only one intermediate value. Following the notation of Section~\ref{notation}, we denote by $(u_{j+1/2}^1)_{j \in \Z}$ the solution after one step of the scheme, and by $(u_{j}^2)_{j \in \Z}$ the solution after the second time step.

\begin{definition}
The \emph{jumps associated to the solution $(u)$} are defined by the sequence $(S_{j}^n = u_{j+1/2}^n- u_{j-1/2}^n)_{j \in \Z}$ for odd steps $n$, and by sequence $(S_{j+1/2}^n = u_{j+1}^n- u_{j}^n)_{j \in \Z}$ for even steps $n$.
\end{definition}

We are interested in the following class of sequences.

\begin{definition} \label{def:EqHypDL}
Let $\alpha$ be a nonnegative real number and $M \geq 1$ be an integer. The set of \emph{$M$-configurations with inner jumps larger than $\alpha$} is the set of jumps
\begin{equation} \label{EqHypDL}\tag{$H_\alpha^M$}
H_\alpha^M= \left\{
\begin{array}{l}
 (v_j)_{j \in \Z} \in \R^\N \text{ such that } \exists j_0 \in \Z: \\
 \qquad \bullet\ v_j=0 \ \text{ if }  j \le j_0\\
 \qquad \bullet\ v_j=1 \ \text{ if }  j \ge j_0+M\\
 \qquad \bullet\ v_{j_0+1}-v_{j_0}>0 \ \text{ and } v_{j_0+M}-v_{j_0+M-1}>0 \\
 \qquad \bullet\ v_{j+1}-v_{j} > \alpha \quad \forall j \in \{j_0+1, \cdots, j_0+ M-2\} 
\end{array}
\right\}
\end{equation}
and the set of \emph{configurations with inner jumps larger than $\alpha$} is $H_\alpha=\bigcup_{M \in \N} H_\alpha^M$.
\end{definition}

From now on we suppose that the initial data $(u_j^0)_{j \in \Z}$ belongs to $H_\alpha^M$ for some $\alpha>0$ and some integer $M$. The long time behavior follows from the following points.
\begin{itemize}
 \item If the initial data is in $H_\alpha$, then all iterations of it also belong to $H_\alpha$. The case $\alpha=0$ is easy (Lemma~\ref{L:middle}), the case $\alpha>0$ requires a finer analysis of the first and last jumps (Lemma~\ref{L:bord}).
 \item The number $M$ of positive jumps essentially decreases with time.
\end{itemize}

For the sake of simplicity, intermediate results are stated at iteration $n=0$. We start with a useful but simple lemma. 

\begin{lemme} \label{L:simple}
 Consider three adjacent cells $u_{j-1}^0$, $u_{j}^0$ and $u_{j+1}^0$, with $u_{j-1}^0 \leq u_{j}^0 \leq u_{j+1}^0$. Denote by $x \mapsto u_{rec}^0(x)$ the associated reconstruction. 
 \begin{itemize}
  \item If $S_{j-1/2}^0 \geq S_{j+1/2}^0$, then
   \begin{equation} \label{controlLS} \int_{j-1/2}^{j} u_{rec}^0(x) dx =  u_j^0-  \frac{u_{j+1}^0}{2}\quad \text{ and } \quad  \int_{j}^{j+1/2} u_{rec}^0(x) dx= \frac{u_{j+1}^0}{2}.
   \end{equation}
   \item If  $S_{j-1/2}^0 \leq S_{j+1/2}^0$, then
 \begin{equation} \label{controlSL} \int_{j-1/2}^{j} u_{rec}^0(x) dx = \frac{u_{j-1}^0}{2} \quad \text{ and } \quad \int_{j}^{j+1/2} u_{rec}^0(x) dx =  u_j^0-  \frac{u_{j-1}^0}{2}.
 \end{equation}
 \end{itemize}
 In any case we have
 \begin{equation} \label{minmaj}
  \frac{u_{j-1}^0}{2} \leq \int_{j-1/2}^{j} u_{rec}^0(x) dx \leq \frac{u_{j}^0}{2} \quad \text{ and } \quad \frac{u_{j}^0}{2} \leq \int_{j}^{j+1/2} u_{rec}^0(x) dx \leq \frac{u_{j+1}^0}{2}.
 \end{equation}
\end{lemme}

As a consequence, the scheme is monotonicity preserving: if the initial sequence $(u_j^0)_{j \in \Z}$ is nondecreasing, so is $(u_{j+1/2}^1)_{j \in \Z}$.

\begin{proof}
The proof is illustrated on Figure~\ref{F:Minimum}. For readability we denote by $a=u_{j-1}^0$,  $b=u_{j}^0$ and $c=u_{j+1}^0$.

If $S_{j-1/2}^0\geq S_{j+1/2}^0$, i.e. if $c-b \leq b-a$, the reconstructed discontinuity lies in the left half cell and its integral is $\frac{b}{2}-  \frac{c-b}{2}  \in \left[ \frac{a}{2},\frac{b}{2} \right]$. On the right half cell the reconstruction is constant equal to $c$.  

If $S_{j-1/2}^0\leq S_{j+1/2}^0$, i.e. if $c-b \geq b-a$, the discontinuity falls in the right half cell and $\int_{j}^{j+1/2} u_{rec}^0(x) dx = \frac{b}{2}+ \frac{b-a}{2} \in \left[  \frac{b}{2} , \frac{c}{2}  \right]$. The reconstruction is equal to $a$ on the left half cell.

 \begin{figure}[h!tp]
 \begin{psfrags}
  \psfrag{a}{$a$}
  \psfrag{b}{$b$}
  \psfrag{c}{$c$}
  \psfrag{d}{$d$}
  \psfrag{a'}{$a'$}
  \psfrag{d'}{$d'$}
  \psfrag{b+c/2}{$\frac{b+c}{2}$}
\psfrag{-1}{}
\psfrag{-1/2}{}
\psfrag{0}{$j-1$}
\psfrag{1/2}{$j-1/2$}
\psfrag{1}{$j$}
\psfrag{3/2}{$j+1/2$}
\psfrag{2}{$j+1$}
 \includegraphics[width=\linewidth]{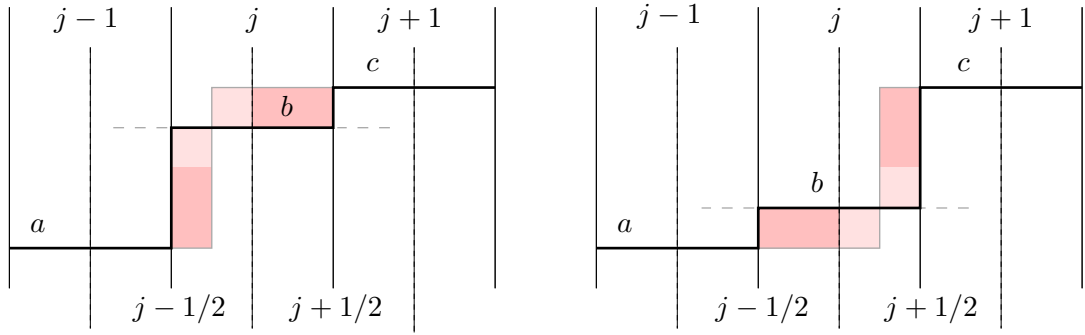}
 \caption{Big jump / small jump (left) or small jump / big jump (right)? Either way, the areas per half cells are easy to compute. The rectangles of the same color have the same area.} \label{F:Minimum}
 \end{psfrags}
\end{figure}
\end{proof}

We now prove that the class $H_\alpha$ of configurations with inner jumps larger than $\alpha$ is preserved by the scheme.

\begin{lemme} \label{L:middle}
If $(u_j^0)_{j \in \Z}$ belongs to $H_0$, so does $(u_{j+1/2}^1)_{j \in \Z}$. More precisely, if the jumps are nonnegative, then
\[ \forall j \in \Z,\ S_{j}^1 \geq \min(S_{j-1/2}^0, S_{j+1/2}^0). \]
\end{lemme}

\begin{proof}
We denote by $u_{j-1}^0=a$, $u_{j}^0=b$ and $u_{j+1}^0=c$, so that $S_{j-1/2}^0=b-a$ and $S_{j+1/2}^0=c-b$.

We begin with the case $S_{j-1/2}^n \geq S_{j+1/2}^n$ (Figure~\ref{F:Minimum}, left). Lemma~\ref{L:simple} gives 
\[ \begin{aligned}
    S_j^1
      &= u_{j+1/2}^1- u_{j-1/2}^1 \\
      &= \int_j^{j+1/2} u^0_{rec}(x) \, dx + \int_{j+1/2}^{j+1} u^0_{rec}(x) \, dx \\
      & \qquad - \int_{j-1}^{j-1/2} u_0^{rec}(x) \, dx  - \int_{j-1/2}^j u_0^{rec}(x) \, dx  \\
      & \geq \frac{c}{2} + \frac{b}{2} -  \frac{b}{2} - \left(b-  \frac{c}{2} \right) = c-b = \min(S_{j-1/2}^0, S_{j+1/2}^0)
   \end{aligned}
\]
A similar computation gives the result in the other case. One can also easily see that $S_{j}^1=0$ for any $j\le 0$ or $j\ge M$.
\end{proof}

This proof does not work immediately for $\alpha>0$ because at the left extremity, it only yields to $S_{j_0}^1\geq \min(S_{j_0-1/2}^0, S_{j_0+1/2}^0)$ which may be non zero and $S_{j_0+1}^1\geq \min(S_{j_0+1/2}^0, S_{j_0+3/2}^0)$, which may be smaller than $\alpha$, as an element of $H_\alpha$ does not have any constraint on the first jump $S_{j_0+1/2}^0$.

\begin{lemme} \label{L:bord}
Suppose that the initial data $(u_j^0)_{j \in \Z}$ belongs to $H_\alpha$ for some $\alpha>0$. Then $(u_{j-1/2}^1)_{j \in \Z}$ also belongs to $H_\alpha$. More precisely,
\begin{itemize}
 \item if $0 < S_{j_0+1/2}^0 \leq S_{j_0+3/2}^0$, i.e. if the first jump is smaller than the second one,  then $S_{j_0}^1=0$;
 \item if  $0 < S_{j_0+3/2}^0 < S_{j_0+1/2}^0$, i.e. if the first jump is larger than the second one, then $0 < S_{j_0}^1 \leq S_{j_0+1}^1$.
Moreover, if $S_{j_0+5/2}^0 \geq \alpha$, then $ S_{j_0+1/2}^2 \leq S_{j_0+1/2}^0 - \frac{\alpha}{4}$.
\end{itemize}
\end{lemme}

\begin{proof}
First of all, from the previous lemma and Hypothesis~\eqref{EqHypDL} we have 
\begin{equation} \label{away}
 \forall j \in \{ j_0+1, \cdots, j_0+M-1 \}, \ S_j^1  \geq \alpha.
\end{equation}
We easily see that $u^0_{rec}=0$ on $(- \infty, j_0+1/2)$ and $u^0_{rec}=1$ on $(j_0+M-1/2, + \infty)$. It yields
$$ \forall j \leq j_0-1 \ \text{ and } \ \forall j \geq j_0+M+1, \ S_j^1=0. $$

We now focus on the two jumps near the left extremities, the results extending trivially to the right extremity. First, if $0 < S_{j_0+1/2}^0 \leq S_{j_0+3/2}^0$, then the reconstruction in $\mathcal{C}_{j_0+1}$ lies in the right half of the cell $\left[ j_0+1, j_0+\frac{3}{2} \right]$. The reconstruction $u^0_{rec}$ is null on $\left[j_0+\frac{1}{2}, j_0+1 \right]$, so $u_{j_0+1/2}^1=0$  and $S_{j_0}^1=0$. The first inner jump is $S_{j_0+2}^1$ and is larger than $\alpha$ by~\eqref{away}.

\begin{figure}[h!tp]
\begin{psfrags}
\psfrag{s-1}{$S_{-1}^0$}
\psfrag{s0}{$S_0^0$}
\psfrag{s1}{$S_1^0$}
\psfrag{s2}{$S_2^0$}
\psfrag{s3}{$S_3^0$}
\psfrag{s-1/2}{$S_{-1/2}^1$}
\psfrag{s1/2}{$S_{1/2}^1$}
\psfrag{s3/2}{$S_{3/2}^1$}
\psfrag{s5/2}{$S_{5/2}^1$}
\psfrag{s-1'}{$S_{-1}^2$}
\psfrag{a}{$a$}
\psfrag{b}{$b$}
\psfrag{c}{$c$}
\psfrag{d}{$d$}
\psfrag{a/2-S2/2}{$\frac{2a-b}{2}$}
\psfrag{3b-c/2}{$\frac{3b-c}{2}$}
\psfrag{a+b/2}{$\frac{a+b}{2}$}
\psfrag{>=a/4}{$ \geq \frac{\alpha}{4}$}
\psfrag{0}{$j_0$}
\psfrag{00}{$0$}
\psfrag{1}{$j_0+1$}
\psfrag{2}{$j_0+2$}
\psfrag{3}{$3$}
\psfrag{1/2}{$j_0+\frac{1}{2}$}
\psfrag{3/2}{$j_0+\frac{3}{2}$}
\psfrag{5/2}{$j_0+\frac{5}{2}$}
\psfrag{>}{$\alpha$}
\psfrag{area1}{area $\leq \frac{a+b}{4}$}
\psfrag{area2}{area $\leq \frac{3b-c}{4}$}
\psfrag{u^0_rec}{$\textcolor{mypurple}{u_{rec}^0}$}
\psfrag{u^1_rec}{$\textcolor{mypurple}{u_{rec}^1}$}
\psfrag{u^0}{$\textcolor{mygreen}{u^0}$}
\psfrag{u^1}{$\textcolor{mygreen}{u^1}$}
 \includegraphics[width=\linewidth]{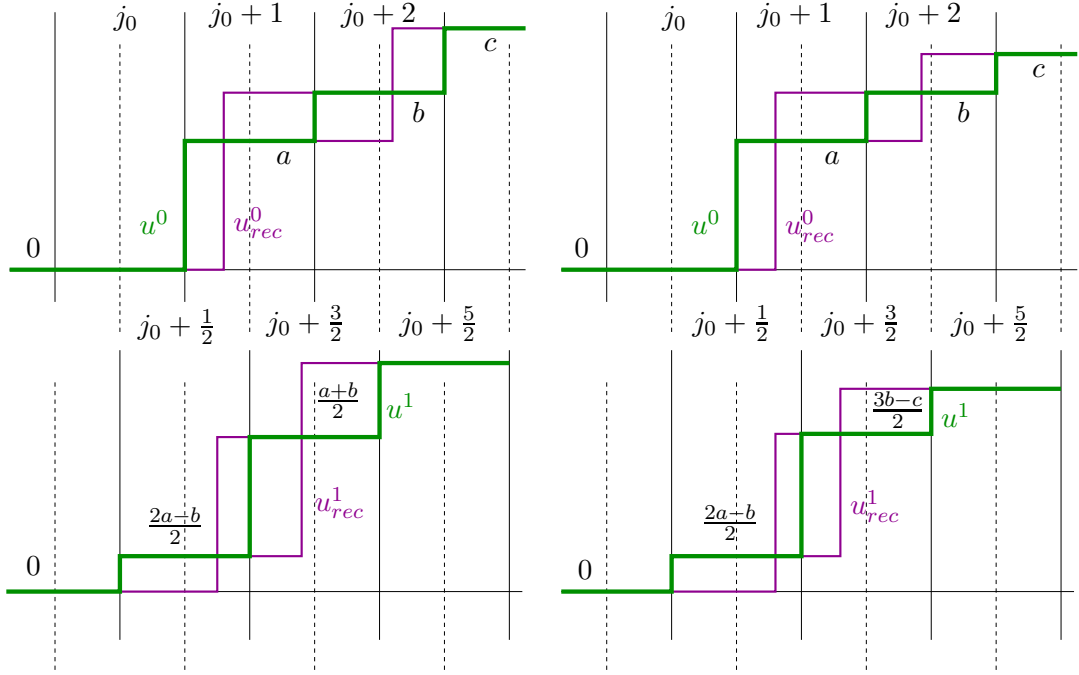}
 \caption{Behavior at the left extremity when the first jump is larger than the second one, depending on the relative sizes of the second and third jumps.} \label{F:LeftExtremity}
\end{psfrags}
\end{figure}

We now focus on the second case $S_{j_0+1/2}^0 > S_{j_0+3/2}^0 \geq \alpha$. In this case, one has on the one hand $S_{j_0}^1>0$ and on the other hand, by Lemma~\ref{L:middle}, $S_{j_0+1}^1 \geq \min(S_{j_0+1/2}^0, S_{j_0+3/2}^0) \geq \alpha$. Hence, $(u_{j-1/2}^1)_j\in H_\alpha$, and it remains to prove the bound about the second iteration. We will do it by a case disjunction, in each case it will follow from a simple computation.

We denote as usual $u_{j_0+1}^0=a$, $u_{j_0+2}^0=b$ and $u_{j_0+3}^0=c$. The elements of proof are illustrated on Figure~\ref{F:LeftExtremity}. Suppose first that the second jump $S_{j_0+3/2}^0$ is smaller than the third jump $S_{j_0+5/2}^0$ (Figure~\ref{F:LeftExtremity}, left). Using Lemma~\ref{L:simple} we obtain
\[ u_{j_0-1/2}^1=0, \ u_{j_0+1/2}^1=\frac{2a-b}{2}\ \text{and} \  \ u_{j_0+3/2}^1= \frac{a+b}{2}. \] 
It follows that
\[ S_{j_0}^1 =  \frac{2a-b}{2} < \frac{a}{2} \ \text{ and } \ S_{j_0+1}^1 = \frac{a+b}{2} - \frac{2a-b}{2} = \frac{2b-a}{2} > \frac{a}{2}.\]
We thus have $S_{j_0}^1< S_{j_0+1}^1$, which implies that $u_{j_0}^2=0$ and $S_{j_0-1/2}^2=0$. We can bound the jump in $j_0+1/2$, using once again Lemma~\ref{L:simple}:
\[ \begin{aligned}
    S_{j_0+1/2}^2 & = u_{j_0+1}^2-u_{j_0}^2 = u_{j_0+1}^2 \\
      &= \int_{j_0+1/2}^{j_0+1} u_{rec}^{1}(x) \, dx + \int_{j_0+1}^{j_0+3/2} u_{rec}^1(x) \, dx \\
      & \leq \frac{2 u_{j_0+1/2}^1-u_{j_0-1/2}^1}{2} + \frac{u_{j_0+3/2}^1}{2} \hspace{2cm}  \text{with~\eqref{controlSL} and~\eqref{minmaj}} \\
      & = a- \frac{b-a}{4}  \leq S_{j_0+1/2}^0 - \frac{\alpha}{4}.
   \end{aligned}
\]

To conclude, we treat the case where the third jump $S_{j_0+5/2}^0$ is smaller than the second jump $S_{j_0+3/2}^0$ (Figure~\ref{F:LeftExtremity}, right). We still have $u_{j_0+1/2}^1=\frac{2a-b}{2}$, and using Lemma~\ref{L:simple}, first case both on $\left[j_0+1, j_0+\frac{3}{2} \right]$ and $\left[j_0+\frac{3}{2}, j_0+2 \right]$ we obtain $u_{j_0+3/2}^1= \frac{3b-c}{2}$. It follows that
\[ S_{j_0}^1=  \frac{2a-b}{2} \ \text{ and } \ S_{j_0+1}^1 = \frac{4b-2a-c}{2}.\]
The second jump is larger than the first jump:
\[ S_{3/2}^0\geq S_{5/2}^0 \ \Longrightarrow \ 4(b-a)>(c-b) \ \Longrightarrow \   \frac{4b-2a-c}{2} \geq \frac{2a-b}{2}. \]
Thus the first point of Lemma~\ref{L:bord} gives $u_{j_0}^2=0$ and $S_{j_0-1/2}^2=0$. Eventually we compute $S_{j_0+1/2}^2= u_{j_0+1}^2$. Using once again~\eqref{controlSL} and~\eqref{minmaj} we obtain
\[ \int_{j_0+\frac{1}{2}}^{j_0+ 1} u_{rec}^1 = u_{j_0+1/2}^1=\frac{2a-b}{2} \ \text{ and } \ \int_{j_0+ 1}^{j_0+\frac{3}{2}} u_{rec}^1 \leq \frac{u_{j_0+3/2}^1}{2}= \frac{3b-c}{4}. \]
We end up with
\[S_{j_0+1/2}^2\leq \frac{4a+b-c}{4}=a - \frac{c-b}{4} \leq S_{j_0+1/2}^0 -\frac{\alpha}{4},\]
which concludes the proof.

\end{proof}

We are now in position to prove the following theorem. It states that  after a finite number of iterations, the process described in Section~\ref{notation} is $2$-periodic and at each time step, the sequence contains at most one intermediate value.

\begin{theoreme} \label{thmover}
Suppose that the initial data $(u_j^0)_{j \in \Z}$ belongs to $H_\alpha^M$ for some $\alpha>0$ and some integer $M>0$. Then there exists an integer $p=p(M,\alpha)$ such that for all $n \geq p$, $(u_j^n)_{j \in \Z}$ is in $H^2_\alpha$ or $H^1_\alpha$, and the solution is $2$-periodic: for all $n \geq p$,  $u_j^{n+2}=u_j^n$.
\end{theoreme}

\begin{proof} 
All along the proof, we make a slight abuse of notation and drop distinction between odd en even iteration in time, always denoting $(u_j^n)_{j \in \Z}$. We do not wish to systematically distinguish between the two cases, which would only makes the notation heavier.

We know by Lemma~\ref{L:bord} that for all iteration in time $n$, there exists an integer $M^n$ such that $(u_j^n)_{n \in \Z}$ belongs to $H_\alpha^{M^n}$. Note that if $M^n=1$ or $M^n=2$, $\alpha$ plays no role in the definition of $H_\alpha^{M}$.
We prove that the number $M^n$ of strictly positive jumps at iteration $n$ reaches~$2$ in a finite number of iterations. If $M^n \leq 2$, then $(u_j^n)_{j \in \Z}$ has only one intermediate value, thus $u_{rec}$ contains a single discontinuity and it is easily seen that this form is $2$-periodic (this is a special case of~\cite{DL02}, Theorem 3).

We now estimate the number $M^n$ of non zero jumps. Following Definition~\ref{def:EqHypDL}, we denote by $j_0^n$ the last cell where $u_j^n$ is null and recall that the first cell where $u_j^n$ is $1$ is $j_0^n+M^n$. Depending on whether the first jump is larger or smaller than the second one, Lemma~\ref{L:bord} gives the relation between $j_0^{n+1}$ and $j_0^n$.
With a similar argument at the right extremity we deduce that $M^{n+1}$ is equal to $M^n+1$,  $M^n$ or  $M^n-1$. 

There are four cases that can occur at the extremities. In what follows, L stands for ``large jump'' and S stands for ``small jump'' (relatively to each other). The elements of proof are gathered on Picture~\ref{F:Automate}, where the points represents the jumps. The inner jumps (in grey) are all larger than $\alpha$. 
\begin{figure}[h!tp]
 \begin{psfrags}
\psfrag{-1}{$M^{n+1}=M^n-1$}
\psfrag{0}{$M^{n+1}=M^n$}
\psfrag{+1}{$M^{n+1}=M^n+1$}
\psfrag{a}{\textcolor{red}{$\alpha$}}
\psfrag{a/4}{\textcolor{red}{$\frac{\alpha}{4}$}}
\psfrag{F}{Finite}
\psfrag{LS...SL}{LS/SL}
\psfrag{LS...LS}{LS/LS}
\psfrag{SL...SL}{SL/SL}
\psfrag{SL...LS}{SL/LS}
 \includegraphics[width=\linewidth]{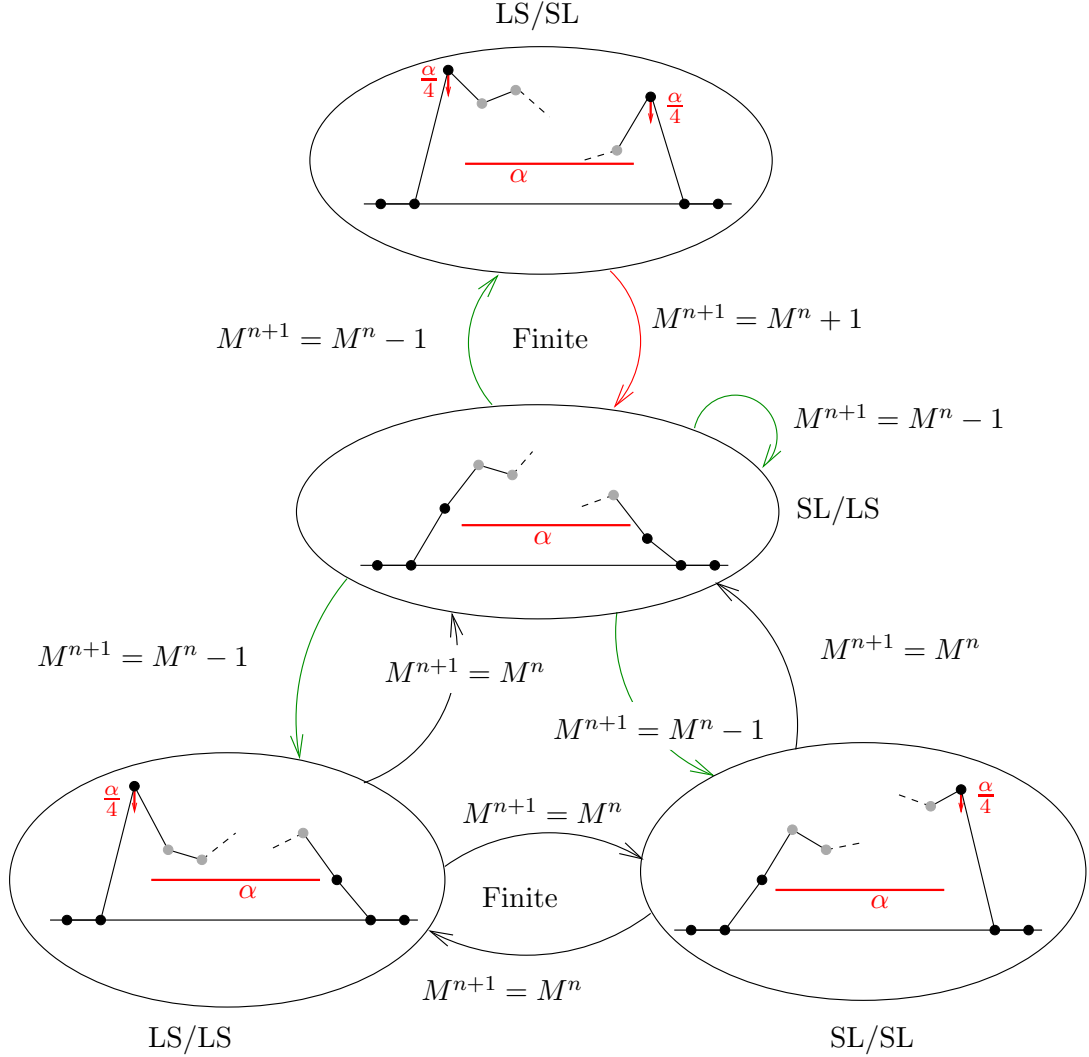}
 \caption{Transitions between the four possible configurations at the extremities and evolution of the number $M^n$ of non null jumps.} \label{F:Automate}
 \end{psfrags}
\end{figure}

\begin{itemize}
 \item \textbf{Case LS/SL, Fig.~\ref{F:Automate}, top:} in that case
 \[ S_{j_0^n+1/2}^n>S_{j_0^n+3/2}^n \ \text{ and } \ S_{j_0^n+M^n-3/2}^n<S_{j_0^n+M^n-1/2}^n, \]
hence (Lemma~\ref{L:bord}, second case)
\[ 0< S_{j_0^{n}}^{n+1} \leq S_{j_0^{n}+1}^{n+1}\ \text{ and } \ 0< S_{j_0^{n}+M^{n}+1}^{n+1} \leq S_{j_0^{n}+M^{n}}^{n+1}.\]
Thus the number of non zero jumps increases by $1$: $M^{n+1}=M^n+1$ and the solution at time $n+1$ is in configuration SL/LS.
 \item \textbf{Case SL/LS, Fig.~\ref{F:Automate}, middle:} in that case
 \[ S_{j_0^n+1/2}^n\leq S_{j_0^n+3/2}^n \ \text{ and } \ S_{j_0^n+M^n-3/2}^n \geq S_{j_0^n+M^n-1/2}^n. \]
 By Lemma~\ref{L:bord}, first case,
 \[  S_{j_0^{n}}^{n+1}=0  \ \text{ and } \ S_{j_0^{n}+M^{n}+1}^{n+1}=0.\]
The number of non zero jumps decreases by $1$:  $M^{n+1}=M^n-1$.  
We do not have any information on  the relative positions of the first and second jump (neither on last and second last) at the iteration $n+1$, and actually the four configurations LS/SL, LS/SL, LS/LS and SL/SL are possible.
 \item \textbf{Case SL/SL, Fig.~\ref{F:Automate}, bottom right:}  in that case
 \[ S_{j_0^n+1/2}^n\leq S_{j_0^n+3/2}^n \ \text{ and } \ S_{j_0^n+M^n-3/2}^n \leq S_{j_0^n+M^n-1/2}^n \]
 and
 \[ 0= S_{j_0^{n}}^{n+1}\ \text{ and } \ 0< S_{j_0^{n}+M^{n}+1}^{n+1} \leq S_{j_0^{n}+M^{n}}^{n+1}.\]
 It follows that $M^{n+1}=M^n$.  At the next iteration in time, we only have information on the last two jumps, so it is possible to end up in situations LS/LS or SL/LS.
 \item \textbf{Case LS/LS, Fig.~\ref{F:Automate}, bottom left:} in that case
 \[ S_{j_0^n+1/2}^n>S_{j_0^n+3/2}^n \ \text{ and } \ S_{j_0^n+M^n-3/2}^n\geq S_{j_0^n+M^n-1/2}^n. \]
 and
\[ 0< S_{j_0^{n}}^{n+1} \leq S_{j_0^{n}+1}^{n+1}\ \text{ and } \ 0= S_{j_0^{n}+M^{n}+1}^{n+1} .\] 
The number of non zero jump remains unchanged $M^{n+1}=M^n$. At time $n+1$, it is possible to be in cases SL/SL and SL/LS.
\end{itemize}

Looking at the transitions between the four possible situations on Figure~\ref{F:Automate}, we see that the number of jumps decreases of $1$ each time Case SL/LS is left. It remains to prove that this is the most frequent case, i.e.that it is not possible to cycle indefinitely from case SL/LS to case LS/SL or from case SL/SL to case LS/LS. It follows from the fact that if $S_{j_0^n+1/2}^n>S_{j_0^n+3/2}^n$, then $S_{j_0^{n+2}+1/2}^{n+2}\leq S_{j_0^n+1/2}^n - \frac{\alpha}{4}$, see Lemma~\ref{L:bord}. Thus, the left extreme value decreases of $\frac{\alpha}{4}$ after one cycle, and in particular will be smaller than $\alpha$ after a finite number of cycles (smaller than $2/\alpha$). At this stage, it is necessary smaller than the first inner jump in $S_{j_0^n+3/2}^n$ which is larger than $\alpha$, thus the solution exits the cycle and falls in situation SL/LS. This concludes the proof.

\end{proof}

\subsection{Half infinite staircase with steps of equal heights}

We now consider staircase-like initial data. It once again illustrates the importance of the behavior at extremities. We say that $(u_j^n)_j$ satisfies Hypothesis~(H') if, up to a horizontal translation,
\begin{enumerate}
\item  $(u_j^n)_j$ is constant on $\{j\le 0\}$, i.e. for any $j\le 0$, $S_{j-1/2}^n=0$;
\item for any $j\ge 3$, we have $S_{j-1/2}^n = 1$;
\item $S_{3/2}^n\ge 1$;
\item $S_{1/2}^n\ge 0$.
\end{enumerate}
As usual, we here considered an even time $n$, the odd case being identical (up to a shift of $1/2$ in the notations).

The following proposition expresses that if the initial condition satisfies Hypothesis~(H'), then the solution will satisfy Hypothesis~(H') at all time, and the total height of the two first steps will tend to infinity as the time goes to infinity.

\begin{prop}\label{PropEscalier}
If $(u_j^n)_j$ satisfies Hypothesis~(H'), then so does $(u_{j-1/2}^{n+1})_j$. Moreover, one has, $S_1^n+S_2^n\to\infty$.
\end{prop}

This proposition is illustrated on Figure~\ref{F:Escalier05}: the first step falls down at each iteration, until it disappears.

\begin{figure}[h!tp]
\begin{center}
 \includegraphics[width=0.7\linewidth]{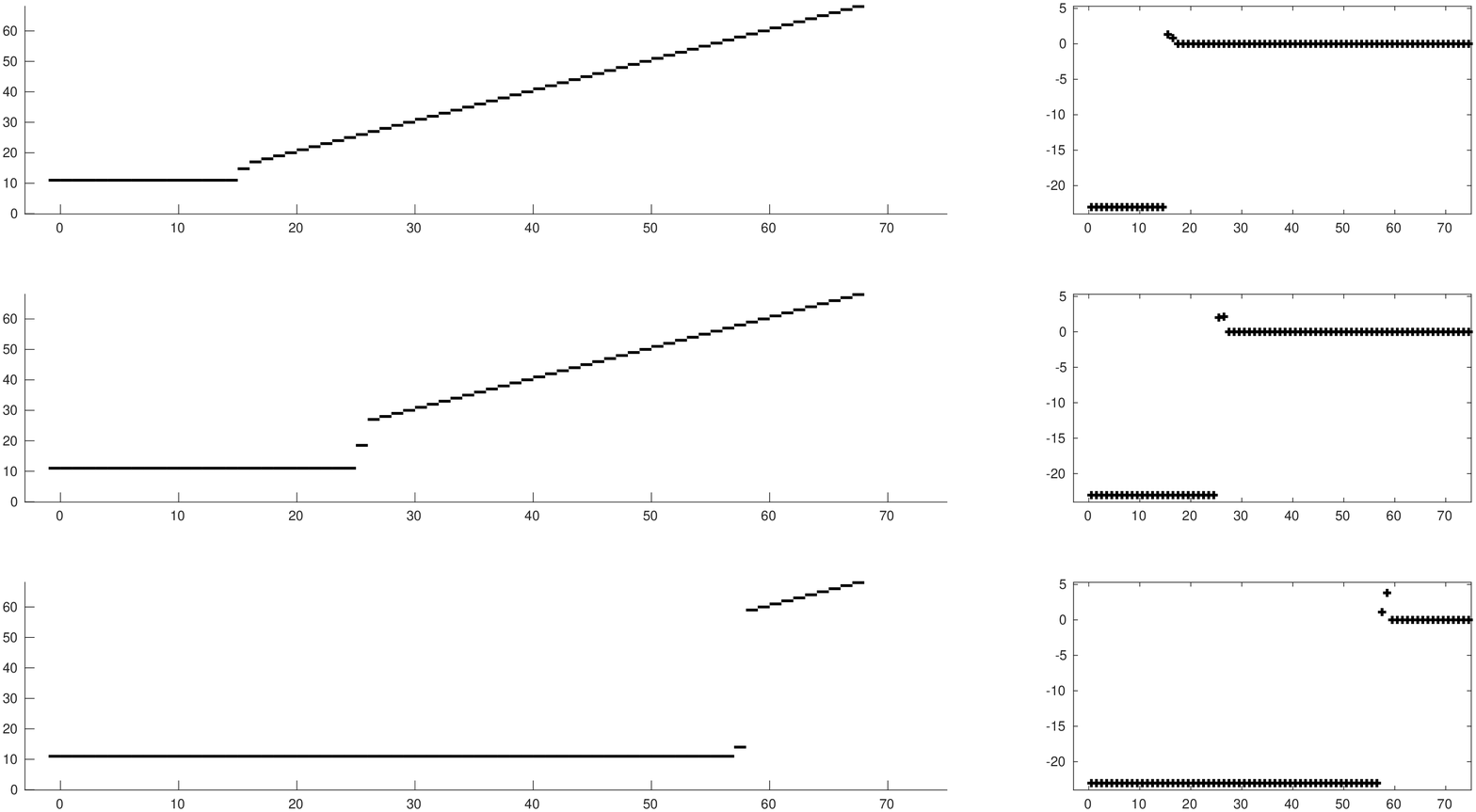}
 \caption{Numerical solution after $100$ iterations (light gray) to $800$ iterations (dark gray) when the initial data is a half infinite staircase} \label{F:Escalier05}
 \end{center}
\end{figure}

\begin{proof}
We first prove the first part of the proposition.

Of course, Hypothesis~(1) of (H') is still satisfied at time $n+1$. As $S_{5/2}^n=1$, Equation \eqref{EqAlpha} gives:
\[d_1^n = \frac{S_{3/2}^n}{S_{1/2}^n+S_{3/2}^n} \qquad \text{and} \qquad d_2^n = \frac{1}{1+S_{3/2}^n}\]
By Hypothesis~(3) of (H'), we have $d_2^n\le 1/2$. Then the expression of the reconstruction depends on the sign of $d_1^n - 1/2$. We have two cases:
\begin{enumerate}[(i)]
\item $d_1^n\le 1/2 \iff S_{1/2}^n\ge S_{3/2}^n$. In this case, a computation leads to
$$ u_{1/2}^{n+1}=\frac{S_{1/2}^n-S_{3/2}^n}{2}, \qquad u_{3/2}^{n+1}=u_2^n-\frac{1}{2},$$
and for $j \geq 2$, $u_{j+1/2}^{n+1}=\frac{u_j^n+u_{j+1}^n}{2}$.
So Hypothesis~(2) of (H') is satisfied at step $n+1$. Moreover
\[S_0^{n+1} =  \frac{S_{1/2}^n-S_{3/2}^n}{2} \ \text{ and } \ S_1^{n+1} = \frac{3S_{3/2}^n+S_{1/2}^n-1}{2}\]
and using that $S_{1/2}^n\ge S_{3/2}^n \ge 1$, we get that $S_1^{n+1} \ge 3/2$, so Hypothesis~(3) is satisfied at step $n+1$.

Remark that in this case, we have 
\begin{equation}\label{eqUndeux}
S_0^{n+1}-S_1^{n+1} = \frac{1-4S_{3/2}^n}{2}\le -\frac32,
\end{equation}
so $S_1^{n+1}<S_2^{n+1}$. In other words if case (i) occurs at time $n$, then it occurs case (ii) at time $n+1$.
\item $d_1^n\ge 1/2 \iff S_{1/2}^n\le S_{3/2}^n$. In this case, 
$$ u_{1/2}^{n+1}=0, \qquad u_{3/2}^{n+1}=\frac{2 u_1^n+u_2^n-1}{2},$$
and for $j \geq 2$, $u_{j+1/2}^{n+1}=\frac{u_j^n+u_{j+1}^n}{2}$. Hypothesis~(2) is immediately satisfied.
Thus we have
$$
S_1^{n+1}  = \frac{3S_{1/2}^n+S_{3/2}^n-1}{2} \ \text{ and } S_2^{n+1} = \frac{S_2^n-S_1^n}{2}+1\ge 1 $$ 
so Hypothesis~(3) is satisfied.
Further computations yields $ S_1^{n+1} + S_2^{n+1} = S_1^n+S_2^n+\frac12$.
\end{enumerate}
\bigskip

Let us now prove the second part of the proposition. A simple computation gives the sum of the two first jumps
\[\left\{\begin{array}{l} 
S_0^{n+1}+S_1^{n+1} = S_{1/2}^n+ S_{3/2}^n - \frac{1}{2}\quad \text{in case (i)}\\
S_1^{n+1}+S_2^{n+1}= S_{1/2}^n+ S_{3/2}^n + \frac{1}{2} \quad \text{in case (ii)}
\end{array}\right.\]
By \eqref{eqUndeux}, if at time $n$ we are in case (i), then at time $n+1$ we have to be in case (ii). So the sequence $(S_1^{2n}+S_2^{2n})_n$ is increasing. To prove that it tends to $+\infty$, we only have to prove that there are infinitely times $m\in\N$ such that at both times $m$ and $m+1$ we are in case (ii).

If at time $n$ we are in case (i), then at time $n+1$ we have to be in case (ii), and a simple computation leads to
\[S_1^{n+2} = S_1^n-\frac34\qquad\text{and}\qquad S_2^{n+2}  = S_2^n+\frac34,\]
in particular
\[S_1^{n+2} - S_2^{n+2} = S_1^n - S_2^n - \frac32.\]
Thus, by a straightforward induction, for any integer $k\le k_0$, where
\[k_0 = \left\lfloor \frac23\big(S_1^n - S_2^n \big)\right\rfloor,\]
we are in case (i) in time $n+2k$ and in case (ii) in time $n+2k+1$, while we are again in case (ii) in time $n+2k+2$. In other words, for any $n\in \N$, we have found a time $m>n$ such that at both times $m$ and $m+1$ we are in case (ii).

\end{proof}

\section{The non symmetric case $\lambda \neq 1/2$} \label{5conf}

Contrary to the last section, we apply the scheme with parameter $\lambda \in (0,1/2)$. The grid is shifted alternatively to the left and to the right as explained in Section~\ref{notation}. We will prove that there is an open set of initial conditions which are $5$-configurations (five jumps, four intermediate values, see Definition~\ref{def:EqHypDL}) on which the solutions converge exponentially to a $5$-configuration having a size $2$ plateau.

We first need some notations. We take $j_0=0$, denote $\varep^n = u_3^n-u_2^n$, and
\begin{align*}
u_1^\infty & = u_1^0 - \frac{2\lambda-\lambda^2}{1-4\lambda^2} \varep^0\\
u_2^\infty = u_3^\infty & = \frac{1+\lambda}{1+2\lambda}u_2^0 + \frac{\lambda}{1+2\lambda}u_3^0\\
u_4^\infty & = u_4^0 + \frac{1-\lambda^2}{1-4\lambda^2}\varep^0
\end{align*}

\begin{prop}\label{PropConvExp}
If $(u_j^0)$ is in a $5$-configuration (meaning that $(S_n^0)\in M_0^5$), and if
\begin{enumerate}[(a)]
\item\label{4conf1} $u_2^0-u_1^0 \ge 2\varep^0$;
\item\label{4conf2} $u_1^\infty \ge \lambda u_2^\infty$;
\item\label{4conf3} $\lambda(u_4^0-u_3^0)\ge (1-\lambda)\varep^0$;
\item\label{4conf4} $u_4^0 - u_3^0 \ge \lambda (1- u_3^0)$; 
\item\label{4conf6} $1- u_4^\infty \ge (1-\lambda^2)\varep^0$;
\end{enumerate}
then $(u_i^n)_i$ is in a $5$-configuration for any $n\ge 0$. Moreover, $(u_i^n)_i$ converges uniformly exponentially fast towards the configuration $(u_i^\infty)_i$.
\end{prop}

This convergence is illustrated on Figure~\ref{F:5Config}

\begin{figure}[h!tp]
\begin{center}
 \includegraphics[width=0.7\linewidth]{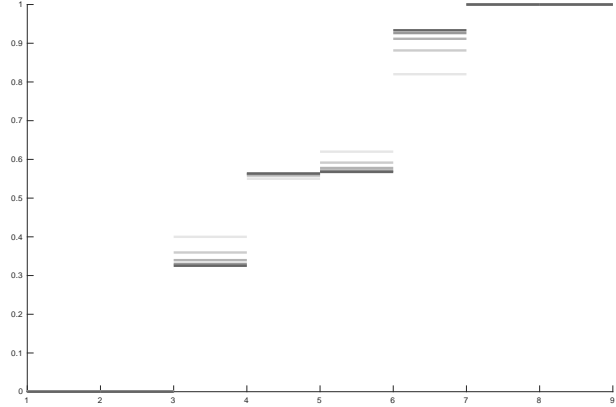}
 \caption{First iterations of the scheme for a $5$-configurations verifying the hypothesis of Proposition~\ref{PropConvExp}. Most recent iterations are of darker color.} \label{F:5Config}
\end{center}
\end{figure}

\begin{figure}[ht]
\begin{center}
\begin{tikzpicture}[scale=1]

\filldraw[color=green!50!black, fill=green!20!white] (1, 2.6) rectangle (2,1.13);
\filldraw[color=green!50!black, fill=green!20!white] (4, 3.26) rectangle (5,4.8);

\draw[very thick, color=black](-0.5, 0.5) -- (1,0.5);
\draw[very thick, color=black](5, 4.8) -- (6.5,4.8);

\draw[very thick, color=black] (1, 1.4) -- (2,1.4);
\draw[very thick, color=black] (2, 2.6) -- (4,2.6);
\draw[very thick, color=black] (4, 3.8) -- (5,3.8);

\draw[<->, color=green!30!black] (2.2,.5) -- (2.2,1.13) node[midway, right]{$\lambda$};
\draw[<->, color=green!30!black] (2.2,1.13) -- (2.2,2.6) node[midway, right]{$1-\lambda$};
\draw[<->, color=green!30!black] (5.2,2.6) -- (5.2,3.263) node[midway, right]{$\lambda$};
\draw[<->, color=green!30!black] (5.2,3.26) -- (5.2,4.8) node[midway, right]{$1-\lambda$};

\foreach \k in {0,...,5}
 {\draw[color=blue!70!black](\k,0) -- (\k,5);
  \draw[color=black](\k+.5,.2) node{$v_{\k}$};}

\draw[color=blue!70!black](6,0) -- (6,5);
\end{tikzpicture}
\caption{\label{Fig4ConfigInfini} One sets $v_0=0$, $v_5=1$ and chooses $v_2=v_3\in(0,1)$. This defines ``good'' intervals (in green) which are the upper parts of the intervals $(0,v_2)$ and $(v_3,1)$ of relative lengths $1-\lambda$; the numbers $v_1$ and $v_4$ can be chosen anywhere in these intervals.}
\end{center}
\end{figure}
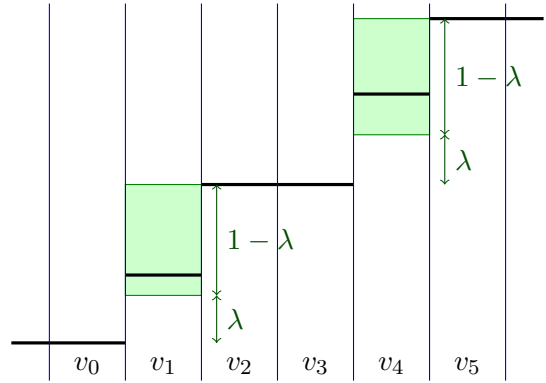

\begin{rem}
 The set of $5$-configurations satisfying \eqref{4conf1} to \eqref{4conf6} contains a nonempty open set. Indeed, consider any number $v_2=v_3\in (0,1)$, and two numbers $v_1$ and $v_4$ such that
\[\frac{v_1-0}{v_2-0}, \frac{v_4-v_2}{1-v_2} \in (\lambda,1)\]
(see Figure \ref{Fig4ConfigInfini}). In other words, consider the intervals $(0,v_2)$ and $(v_2,1)$ and divide each of them into two intervals of relative sizes $\lambda$ and $1-\lambda$; the numbers $v_1$ and $v_4$ have to be in the upper respective subintervals. One easily checks that any $5$-configuration sufficiently close to the configuration $(v_i)$ satisfies \eqref{4conf1} to \eqref{4conf6}.
\end{rem}

\begin{rem}
With a small computation one can see that condition~\eqref{4conf2} cannot hold when $\lambda=1/2$ unless $\eps^0$ is null. Thus this result is not in contradiction with Section~\ref{lambda=1/2}.
\end{rem}

\begin{proof}
We prove by induction the following properties:
\begin{itemize}
\item $(u_j^{2n})$ is in a $5$-configuration;
\item $\varep^{2n} = (4\lambda^2)^n \varep^0$;
\item and the following bounds on the intermediate values:
\begin{align*}
u_1^0 - (2\lambda-\lambda^2)\frac{1-(4\lambda^2)^{n+1}}{1-4\lambda^2} \varep^0 & \le u_1^{2n} \le u_1^0\\
u_2^0 & \le u_2^{2n} \le u_2^0 + (\lambda-2\lambda^2)\frac{1-(4\lambda^2)^{n+1}}{1-4\lambda^2} \varep^0\\
u_3^0 - (1-\lambda-2\lambda^2)\frac{1-(4\lambda^2)^{n1}}{1-4\lambda^2} \varep^0 & \le u_3^{2n} \le u_3^0\\
u_4^0 & \le u_4^{2n} \le u_4^0 + (1-\lambda^2)\frac{1-(4\lambda^2)^{n+1}}{1-4\lambda^2} \varep^0.
\end{align*}
\end{itemize}
Configurations satisfying these properties are said to satisfy property (P). In particular, this will prove that the configurations $u_j^{2n}$ are as in Figure \ref{Fig4Config}, since these conditions imply that
\begin{align*}
u_1^\infty & \le u_1^{2n} \le u_1^0\\
u_2^0 & \le u_2^{2n} \le u_2^\infty\\
u_3^\infty & \le u_3^{2n} \le u_3^0\\
u_4^0 & \le u_4^{2n} \le u_4^\infty,
\end{align*}

Suppose that a configuration $u_j^{2n}$ satisfies property (P). We want to prove that the configuration $u_j^{2n+2}$ still satisfies property (P).

For odd iteration in time, the grid is shifted to the left and the distance $d_j^{2n}=\dfrac{u_j^{2n}-u_{j-1}^{2n}}{u_{j+1}^{2n}-u_{j-1}^{2n}}$ satisfies
\[ d_j^{2n} \geq \lambda \quad \Longleftrightarrow (1-\lambda)(u_j^{2n}-u_{j-1}^{2n}) \geq \lambda (u_{j+1}^{2n}-u_{j}^{2n}). \]
Let us prove that $d_1^{2n} \ge \lambda$, $d_2^{2n} \ge \lambda$, $d_3^{2n} \le \lambda$ and $d_4^{2n} \ge \lambda$. By the hypotheses made on the initial configuration, we have respectively
\begin{itemize}
\item $(1-\lambda)(u_1^{2n}-0) \ge \lambda (u_2^{2n}- u_1^{2n})$ because by Condition \eqref{4conf2},
\[(1-\lambda) u_1^{2n} \geq (1-\lambda) u_1^\infty \geq \lambda (u_2^\infty-u_1^\infty) \geq \lambda (u_2^{2n}-u_1^{2n}).\]

\item $(1-\lambda)(u_2^{2n}-u_1^{2n}) \ge \lambda\varep^{2n}$ which is true by Condition \eqref{4conf1}:
\[ (1-\lambda) (u_2^{2n}-u_1^{2n}) \geq (1-\lambda) (u_2^{0}-u_1^{0})  \geq \lambda \eps^0 \geq \lambda \eps^{2n}.\]

\item $(1-\lambda)\varep^{2n} \le \lambda(u_4^{2n}-u_3^{2n})$ which is true by Condition \eqref{4conf3}.
\item $(1-\lambda)(u_4^{2n}-u_3^{2n}) \ge \lambda(1-u_4^{2n})$ which is true by Condition \eqref{4conf4}.
\end{itemize}

In these cases, one can compute $u_{i-\lambda}^{2n+1}$ (using in particular \eqref{EqAlpha}). It is a $5$-configuration, with
\begin{align*}
u_{1-\lambda}^{2n+1} & = u_1^{2n} - \lambda u_2^{2n}\\
u_{2-\lambda}^{2n+1} & = u_2^{2n} - \lambda \varep^{2n}\\
u_{3-\lambda}^{2n+1} & = u_2^{2n} + \lambda \varep^{2n}\\
u_{4-\lambda}^{2n+1} & = u_4^{2n} - \lambda(1-u_2^{2n}) + \varep^{2n}
\end{align*}
and in particular, 
\[\varep^{2n+1} = 2\lambda \varep^{2n}\]

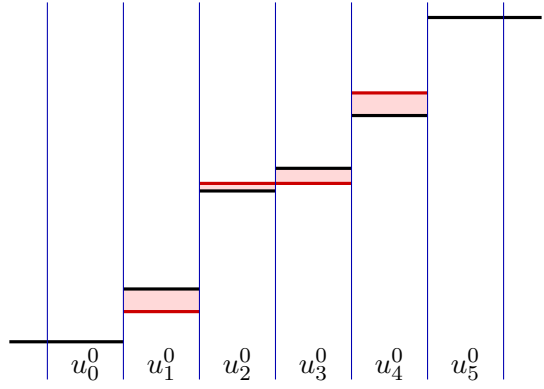
\begin{figure}[ht]
\begin{center}
\begin{tikzpicture}[scale=1]

\fill[color=red!15!white] (1, 1.2) rectangle (2,.9);
\fill[color=red!15!white] (2, 2.5) rectangle (3,2.6);
\fill[color=red!15!white] (3, 2.8) rectangle (4,2.6);
\fill[color=red!15!white] (4, 3.5) rectangle (5,3.8);

\draw[very thick, color=black](-0.5, 0.5) -- (1,0.5);
\draw[very thick, color=black](1, 1.2) -- (2,1.2);
\draw[very thick, color=black](2, 2.5) -- (3,2.5);
\draw[very thick, color=black](3, 2.8) -- (4,2.8);
\draw[very thick, color=black](4, 3.5) -- (5,3.5);
\draw[very thick, color=black](5, 4.8) -- (6.5,4.8);

\draw[very thick, color=red!80!black] (1, .9) -- (2,.9);
\draw[very thick, color=red!80!black] (2, 2.6) -- (4,2.6);
\draw[very thick, color=red!80!black] (4, 3.8) -- (5,3.8);

\foreach \k in {0,...,5}
 {\draw[color=blue!70!black](\k,0) -- (\k,5);
  \draw[color=black](\k+.5,.2) node{$u_{\k}^0$};}

\draw[color=blue!70!black](6,0) -- (6,5);
\end{tikzpicture}
\caption{\label{Fig4Config}The solutions $u_i^{2n}$ all lie in the red domains delimited by the initial configuration $u_i^0$ (black) and the limit configuration $u_i^\infty$ (red) having a 2 plateau}
\end{center}
\end{figure}

For the next iteration in time, the grid is shifted to the right, the distance $d_{j-\lambda}^{2n+1}=\dfrac{u_{j+1-\lambda}^{2n+1}-u_{j-\lambda}^{2n+1}}{u_{j+1-\lambda}^{2n+1}-u_{j-1-\lambda}^{2n+1}}$ is counted from the left interface and
\[ d_{j-\lambda}^{2n+1} \geq \lambda \quad \Longleftrightarrow (1-\lambda)(u_{j+1-\lambda}^{2n+1}-u_{j-\lambda}^{2n+1}) \geq \lambda (u_{j-\lambda}^{2n+1}-u_{j-1-\lambda}^{2n+1}). \]

Now, we have $d_{1-\lambda}^{2n+1}  \ge\lambda$, $d_{2-\lambda}^{2n+1}  \le \lambda$, $d_{3-\lambda}^{2n+1}  \ge \lambda$ and $d_{4-\lambda}^{2n+1}  \ge \lambda$, because by Hypothesis~(P), we have respectively
\begin{itemize}
\item $\lambda (u_{1-\lambda}^{2n+1}-0) \le (1-\lambda) (u_{2-\lambda}^{2n+1}-u_{1-\lambda}^{2n+1}) \iff u_2^{2n}-u_1^{2n} \ge (\lambda-\lambda^2)\varep^{2n}$ which is true by condition \eqref{4conf1};
\item $\lambda(u_{2-\lambda}^{2n+1} -u_{1-\lambda}^{2n+1}) \ge (1-\lambda)\varep^{2n+1} \iff u_2^{2n}-u_1^{2n} + \lambda u_2^{2n} \ge (2-\lambda)\varep^{2n}$ which is true by condition \eqref{4conf1};
\item $\lambda \varep^{2n+1}  \le (1-\lambda)(u_{4-\lambda}^{2n+1} -u_{3-\lambda}^{2n+1} ) \iff \lambda(1-\lambda)(1-u_4^{2n}) \le (2-4\lambda)\varep^{2n} + (1-\lambda)^2 (u_4^{2n}-u_3^{2n})$ which is true by condition \eqref{4conf4};
\item $\lambda(u_{4-\lambda}^{2n+1}- u_{3-\lambda}^{2n+1}) \le (1-\lambda)(1-u_{4-\lambda}^{2n+1})    \iff (1-\lambda^2)\varep^{2n} \le 1-u_4^{2n}$ which is true by condition \eqref{4conf6}.
\end{itemize}
These conditions allow to compute the sequence $u_j^{n+2}$:
\begin{align*}
u_1^{2n+2} & = u_1^{2n+1} + \lambda u_2^{2n+1} - (1-\lambda) \varep^{2n+1}\\
u_2^{2n+2} & = u_3^{2n+1} - \lambda \varep^{2n+1}\\
u_3^{2n+2} & = u_3^{2n+1} + \lambda \varep^{2n+1}\\
u_4^{2n+2} & = u_4^{2n+1} + \lambda(1-u_3^{2n+1})
\end{align*}
thus
\begin{align*}
u_1^{2n+2} & = u_1^{2n} - (2\lambda-\lambda^2) \varep^{2n}\\
u_2^{2n+2} & = u_2^{2n} + (\lambda-2\lambda^2) \varep^{2n} \\
u_3^{2n+2} & = u_3^{2n} - (1-\lambda-2\lambda^2) \varep^{2n}\\
u_4^{2n+2} & = u_4^{2n} + (1-\lambda^2) \varep^{2n}
\end{align*}
and in particular
\[\varep^{2n+2} = 4\lambda^2 \varep^{2n}.\]
\end{proof}

\begin{figure}[h!tp]
\begin{center}
 \includegraphics[clip=true, trim=4.5cm 0cm 4.5cm 0cm, width=1\linewidth]{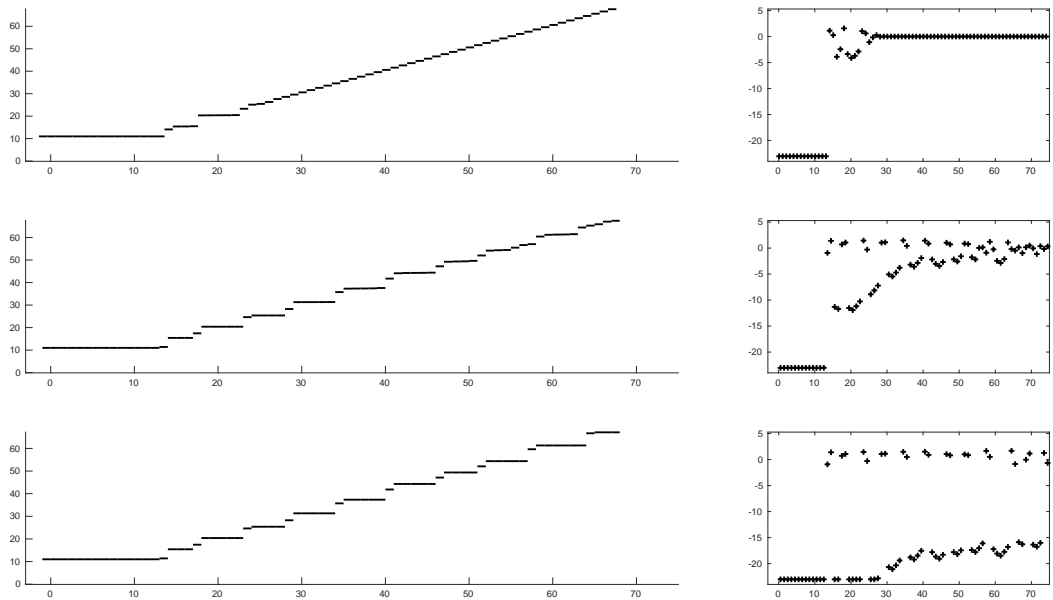}
 \caption{Left: apparition of large plateaus when the initial data is a half infinite staircase and $\lambda=0.4$. Right: height of the jumps in log-scale; we observe an exponential convergence as in Proposition~\ref{PropConvExp}} \label{F:expconv} \label{F:convexp}
\end{center}
\end{figure}

As a conclusion, let us mention that on general initial data, the sequence $(u_j^n)_{j \in \R}$ quickly goes from its initial state to some ``stairshaped'' organization.
Then, we observe an exponential convergence of the smaller jumps, as illustrated on Figure~\ref{F:convexp}.

\bibliographystyle{smfalpha}
\bibliography{biblioAntidiffusif}

\end{document}